\documentclass[12pt]{amsart}

\usepackage{todonotes}   

\usepackage{amsmath,amssymb,amsthm}

\usepackage[ cal= cm, bb = ams, frak = euler, scr = rsfs]{mathalpha}

\usepackage{microtype}

\usepackage[T1]{fontenc}

\usepackage[utf8]{inputenc}

\usepackage{verbatim}

\usepackage{float}

\usepackage{xcolor}

\usepackage[ unicode=true, pdfusetitle, bookmarks=true,
bookmarksnumbered=false, breaklinks=false, backref=false,
colorlinks=true, linkcolor=blue!70!black, citecolor=black,
urlcolor=blue!78!red, pagebackref, final]{hyperref}

\usepackage{graphicx}

\usepackage{grffile}

\usepackage{array}

\usepackage[section]{placeins}

\usepackage{cancel}

\usepackage[capitalise]{cleveref}

\newcommand{\clevertheorem}[3]{%
  \newtheorem{#1}[thm]{#2} \crefname{#1}{#2}{#3} }

\theoremstyle{plain} \newtheorem{thm}{Theorem}[section]
\crefname{thm}{Theorem}{Theorems} \newtheorem*{thm*}{Theorem}
\clevertheorem{prop}{Proposition}{Propositions}
\newtheorem*{prop*}{Proposition} \clevertheorem{lem}{Lemma}{Lemmas}
\clevertheorem{cor}{Corollary}{Corollaries}
\clevertheorem{conj}{Conjecture}{Conjectures}

\theoremstyle{definition}
\clevertheorem{defn}{Definition}{Definitions}
\clevertheorem{defns}{Definitions}{Definitions}
\clevertheorem{question}{Question}{Questions}
\clevertheorem{example}{Example}{Examples}
\clevertheorem{exercise}{Exercise}{Exercises}
\clevertheorem{notation}{Notation}{Notations}
\clevertheorem{rk}{Remark}{Remarks}

\theoremstyle{remark} 

\makeatletter \providecommand\@dotsep{5} \makeatother

\usepackage{ifdraft}

\ifdraft{
  \usepackage[top=.7in,bottom=.7in,left=1in,right=1in]{geometry}
  
  \usepackage[datesep={-},showseconds=false]{datetime2}
  \DTMsettimestyle{iso}
  
  \usepackage{draftwatermark} \SetWatermarkScale{1}
  \SetWatermarkAngle{90} \SetWatermarkLightness{.75}

  \SetWatermarkText{\shortstack{Draft : \today ~at~ \DTMcurrenttime
      \vspace{45pc}}} } {
  \usepackage[top=.7in,bottom=.7in,left=.8in,right=.8in]{geometry}

  \PassOptionsToPackage{disable}{todonotes} }
  
\usepackage{cancel}
\usepackage{tikz} \usetikzlibrary{decorations.markings}
\usepackage{tikz-cd}

\hypersetup{ pdfauthor={Nathan Fieldsteel}, pdftitle={Minimal
    resolutions for modules over polynomial OI-algebras},
  pdfkeywords={}, pdfsubject={}, pdfcreator={Emacs 26.1},
  pdflang={English}}

\author{Nathan Fieldsteel} 
\address{Department of Mathematics,
  University of Kentucky, 715 Patterson Office Tower,  Lexington, KY 40506 USA}
\email{nathan.fieldsteel@uky.edu}
\urladdr{\url{http://nathanfieldsteel.github.io}}

\author{Uwe Nagel}
\address{Department of Mathematics, University of
  Kentucky, 715 Patterson Office Tower, Lexington, KY 40506 USA} 
\email{uwe.nagel@uky.edu}
\urladdr{\url{http://www.ms.uky.edu/~uwenagel/}}

\title{Minimal and Cellular Free Resolutions over Polynomial OI-algebras}

\newcommand{\FI}{\mathrm{FI}}
\newcommand{\OI}{\mathrm{OI}} 
\newcommand{\eps}{\varepsilon}
\newcommand{\N}{\mathbb{N}}
\newcommand{\ffi}{\varphi}

\newcommand{\A}{\mathbf{A}}
\newcommand{\B}{\mathbf{B}}
\newcommand{\bF}{\mathbf{F}}
\newcommand{\Ib}{\mathbf{I}}

\newcommand{\M}{\mathbf{M}}

\newcommand{\1}{\mathbf{1}}

\DeclareMathOperator{\rank}{rank}
\DeclareMathOperator{\sgn}{sgn}

\newcommand{\Z}{\mathbb{Z}}

\begin{document}

\begin{abstract}
  Minimal free resolutions of graded modules over a noetherian
  polynomial ring have been attractive objects of interest for more
  than a hundred years. We introduce and study two natural extensions
  in the setting of graded modules over a polynomial \(\OI\)-algebra,
  namely \textit{minimal} and \textit{width-wise minimal} free
  resolutions. A minimal free resolution of an \(\OI\)-module can be
  characterized by the fact that the free module in every fixed
  homological degree, say \(i\), has minimal rank among all free
  resolutions of the module. We show that any finitely generated
  graded module over a noetherian polynomial \(\OI\)-algebra admits a
  graded minimal free resolution and that it is unique. A width-wise
  minimal free resolution is a free resolution that provides a minimal
  free resolution of a module in every width. Such a resolution is
  necessarily minimal. Its existence is not guaranteed. However, we
  show that certain monomial \(\OI\)-ideals do admit width-wise
  minimal free or, more generally, width-wise minimal flat
  resolutions. These ideals include families of well-known monomial
  ideals such as Ferrers ideals and squarefree strongly stable ideals.
  The arguments rely on the theory of cellular resolutions.

\end{abstract}

\thanks{The second author was partially supported by Simons Foundation grants \#317096 and \#636513. }

\maketitle

\tableofcontents


\section{Introduction}\label{intro}  

\(\FI\)-modules were introduced in \cite{MR3357185}. Since then their
theory has been further developed in many directions, see, e.g.,
\cite{MR3654111, MR3285226, MR3084430, MR3418745, MR3556290} . Here
\(\FI\) denotes the category of finite sets and injective functions,
and an \(\FI\)-module is a functor from \(\FI\) to \(R\)-mod where
\(R\) is a commutative ring. Central to the theory of \(\FI\)-modules
is the notion of noetherianity, which encodes the observed
stability. Some of the proofs in the theory of \(\FI\)-modules are
predicated on the introduction of \(\OI\)-modules, where \(\OI\) is
the category of totally-ordered finite sets with order-preserving
injective functions. Every \(\FI\)-module can be thought of as an
\(\OI\)-module by composition with the forgetful functor
\(\OI \rightarrow \FI\).

A separate but related area of interest involves the study of a
non-noetherian ring \(R\) equipped with an action of a group or monoid
\(G\), the prime example being a polynomial ring
\(k[x_{1}, x_{2}, \ldots, x_{n}, \ldots]\) in infinitely many
variables, equipped with an action of the infinite symmetric group.
For suitable \(G\) it is shown that every \(G\)-invariant ideal is
generated by finitely many \(G\)-orbits (see \cite{C, MR2327026,
  MR2854168}).

These two subjects were unified in \cite{1710.09247} with the
introduction of \(\FI\)-and \(\OI\)-algebras and modules over these
algebras, together with a theory of Gr\"obner bases and free
resolutions. Hilbert functions of \(\OI\)-modules were studied in
\cite{MR3666212, {MR3778142}, N-hilb}.

In this paper, we add to the theory of ideals and modules over
OI-algebras by first introducing notions of \textit{minimal} free
resolutions of modules over polynomial OI-algebras. We prove that if
\(\A\) is a noetherian graded polynomial OI-algebra and \(\M\)
is a finitely generated \(\A\)-module, then a minimal graded
free resolution of \(\M\) exists and is unique up to isomorphism, and
any finitely generated free resolution of \(\M\) contains a minimal
free resolution as a direct summand (see \Cref{thm:char min
  res}). This is directly analogous to the theory of minimal free
resolutions for graded modules over noetherian polynomial rings.

In addition, we provide a related notion of a \textit{width-wise
  minimal} free resolution, corresponding to the situation where a
free resolution of an \(\A\)-module \(\M\) in fact parametrizes a
family of minimal free resolutions of modules over different
polynomial rings. This is a much stricter condition, and it is not as
well behaved as ordinary minimality for free resolutions.

We also consider flat resolutions. If an \(\OI\)-module \(\M\) admits
a resolution by \(\OI\)-modules that gives a minimal free resolution
in each width \(\M(w)\) of \(\M\) with \(w \geqslant 0\), then such a
resolution is a resolution by flat \(\OI\)-modules (see
\Cref{prop:char widthwise minimal}). There are modules that do not
admit a width-wise minimal free resolution, but do have a width-wise
minimal flat resolution.
In particular we prove that if \(\Ib\) is an \(\OI\)-ideal
parametrizing a family of Ferrers ideals or of certain squarefree
strongly stable ideals, then the complex-of-boxes resolutions
introduced in \cite{MR2515766} can be assembled into a resolution of
\(\Ib\) by flat \(\A\)-modules, providing an example of a width-wise
minimal flat resolution (see \Cref{sss}). In some cases, this
resolution is even a width-wise minimal free resolution (see
\Cref{cor:minimal free resolution}).

The methods in this note allow one also to establish that any finitely
generated \(\FI\)-module over a noetherian polynomial \(\FI\)-algebra
admits a minimal resolution of free \(\FI\)-modules. However, we are
not aware of any non-free \(\FI\)-module that admits a width-wise
minimal resolution by flat \(\FI\)-modules. Thus, we focus on
resolutions of \(\OI\)-modules.

In \cref{prelims}, we introduce and define our objects of study and
establish some basic properties. In \cref{min}, we define minimal free
resolutions and width-wise minimal free resolutions as well as
width-wise minimal flat resolutions, and we connect their properties
to the existing literature. In particular, we characterize graded flat
modules over a polynomial \(\OI\)-algebra (see \Cref{flatfree}).

Using cellular resolutions, we give explicit constructions of
resolutions of modules over polynomial \(\OI\)-algebras in
\cref{resolutions}.


\section{Preliminaries}\label{prelims} 

We recall needed concepts, fix notation and present illustrating
examples.

\begin{defn}\label{oi}
  We use \(\OI\) to denote the category whose objects are
  totally-ordered finite sets and whose morphisms are the
  order-preserving injective functions.
\end{defn} 

For any natural number \(n \geqslant 0\), there is (up to isomorphism)
only one object in OI with \(n\) elements, namely the interval
\([n] = \left\{i \in \mathbb{Z} ~|~ 1 \leqslant i \leqslant
  n\right\}\).  By an abuse of notation, we will also use \(n\) to
refer to this object in \(\OI\).
  
The category \(\OI\) is equivalent to its skeleton, the category with
just one object \(n\) for each \(n \geqslant 0\) and morphisms being
order-preserving injective maps \(\eps\colon m \to n\). Such a
morphism will be denoted with the string of natural numbers
\(\eps(1)\eps(2)\cdots\eps(m)\). For example, the morphism
\(\eps:3 \rightarrow 8\) defined by its values \(1 \mapsto 3\),
\(2 \mapsto 5\) and \(3 \mapsto 7\) will be written as the string
\(357\). While this can lead to distinct morphisms in OI being
represented by the same string of digits, we will only use this
notation in contexts where no confusion or ambiguity can arise.

\begin{defn}\label{oimod}
  For a commutative unital ring \(k\), an \(\OI\)\textit{-module over
    \(k\)} is a covariant functor from \(\OI\) to the category of
  \(k\)-modules, and a \textit{morphism} between two \(\OI\)--modules
  is a natural transformation. If \(k\) is a graded ring, then a
  \textit{graded} OI\textit{-module over \(k\)} is a covariant functor
  from OI to the category of graded \(k\)-modules with
  degree-preserving maps. When the ring \(k\) or its grading are
  irrelevant or clear from context, we refer to these objects as
  OI\textit{-modules.}
\end{defn}

In order to define a functor \(F\) from \(\OI\) to a category \(C\),
it is enough to define it on the skeleton of \(\OI\). We use this
convention throughout the paper.

In what follows, we will use plain typeface capital letters such as
\(M\) and \(N\) for ordinary modules over a commutative ring, and
boldface capital letters such as \(\M\) and \(\mathbf{N}\) for
\(\OI\)-modules. For the evaluation of an OI-module \(\M\) at an
object \(w\) in \(\OI\), we will use the notation \(\M(w)\). In this
context, the \(\OI\) object \(w\) is called the \textit{width}, and we
refer to \(\M(w)\) as the \textit{width} \(w\) \textit{component} of
\(\M\). An \textit{element} of an \(\OI\)-module \(\M\) is an element
of \(\M(w)\) for some width \(w\). It is said to be an \textit{element
  of width} \(w\).  An \(\OI\)\textit{-submodule} of an \(\OI\)-module
\(\M\) is a subfunctor \(\mathbf{N} \subseteq \M\). A collection \(G\)
of elements of an \(\OI\)-module \(\M\) is called a \textit{generating
  set} for \(\M\) if the smallest submodule of \(\M\) that contains
all elements of \(G\) is \(\M\) itself.



\begin{example}\label{const}
  Any \(k\)-module \(M\) defines a constant functor
  \(\OI \rightarrow k\text{-mod}\), which is an
  \(\OI\)-module. The constant \(\OI\)-module defined by \(k\) itself
  will be denoted \(\mathbf{k}\).
\end{example}

The most important \(\OI\)-modules for our purposes are the
\textit{free} \(\OI\)-modules. See Definition 3.16 in
\cite{1710.09247}, and Definition 2.2 in \cite{MR3285226} for a
definition of the closely-related free \(\FI\)-modules.

\begin{example}
  \label{freerank1}
  Let \(n\) be a natural number and consider the functor
  \[
    \bF^{\OI,n}:\OI\rightarrow k\text{-mod},
  \] 
  which for any width \(w\) gives the free \(k\)-module with basis
  indexed by the \(\OI\) morphisms from \(n\) to \(w\),
  \[\bF^{\OI,n}(w) = \bigoplus\limits_{\pi \in
      \operatorname{Hom}_{\OI}(n,w)} k e_{\pi},\] and where, for
  any morphism \(\eps \in \operatorname{Hom}_{\OI}(w,\ell)\),
  the map
  \(\bF^{\OI,n}(\eps) : \bF^{\OI,n}(w) \rightarrow
  \bF^{\OI,n}(\ell)\) is defined on basis elements by composing
  the indexing morphism with \(\eps\), i.e. by
  \(e_{\pi} \mapsto e_{\eps \circ \pi}.\)
  
  The object
  \(\bF^{\OI,n}\) is called the \textit{free
  }OI\textit{-module of rank \(1\) generated in width \(n\)}. This
  terminology is justified by the observations that
  \(\bF^{\OI,n}\) is generated by the single basis
  element \(e_{\text{id}}\) in width \(n\), and that a map of
  \(\OI\)-modules \(\bF^{\OI,n} \rightarrow \M\)
  is determined by its value on this generating element. In other
  words, there is a natural isomorphism
  \[\operatorname{Nat}(\bF^{\OI,n},\M) \simeq
    \operatorname{Hom}_{k}(k,\M(n))\]
  
\end{example}

The category of \(\OI\)-modules over \(k\) is an abelian category. It
has pointwise direct sum and tensor product operations, which for any
\(\OI\)-modules \(\M\) and \(\mathbf{N}\) are defined in the natural
way by
\[(\M\oplus\mathbf{N})(w) := \M(w) \oplus
  \mathbf{N}(w)\]
and
\[(\M\otimes\mathbf{N})(w) := \M(w) \otimes_{k}
  \mathbf{N}(w).\]

\begin{defn}
  \label{free}

  Generalizing \cref{freerank1}, we say that an \(\OI\)-module is a
  \textit{free} \(\OI\)\textit{-module of rank} \(r\) if it is
  isomorphic to a direct sum
  \(\bF^{\OI,n_{1}}\oplus \ldots \oplus \bF^{\OI,n_{r}}\) for some
  natural numbers \(n_{1}, \ldots, n_{r} \geqslant 0\). That this rank
  is well-defined might not immediately be apparent, but it follows as
  a special case of \cref{oinak}.
\end{defn}

We are interested in a generalization of \(\OI\)-modules, where
instead of one underlying ring \(k\), we have a family of rings
parametrized by \(\OI\), an object called an \(\OI\)-algebra. A module
over an \(\OI\)-algebra is a collection of modules over this family of
rings, with appropriate structure maps between them. These objects
were originally introduced in \cite{1710.09247}, along with their
\(\FI\) counterparts. We recall these concepts. The reader should be
aware that our definitions differ slightly from those which originally
appeared in \cite{1710.09247}, and we will address and justify these
differences as they arise.

\begin{defn}
  An \(\OI\)\textit{-algebra over \(k\)} is a covariant functor \(\A\)
  from \(\OI\) to the category of \(k\)-algebras. An \(\OI\)-algebra
  over \(k\) is called \textit{graded}, \textit{commutative} or
  \textit{graded-commutative} if it takes values in the category of
  \(k\)-algebras with the corresponding property.
\end{defn}

\begin{rk}
  This is a a slight relaxation of the definition given in
  \cite{1710.09247}, because here we do not insist that the width
  \(0\) algebra \(\A(0)\) is \(k\) itself. Because of this omission,
  our definition is equivalent to the category-theoretic definition
  than an \(\OI\)-algebra over \(k\) is a monoid in the category of
  \(\OI\)-modules over \(k\), where \(\mathbf{k}\) is the unit
  object. This change is inconsequential for our purposes, since
  whenever \(\A\) is an \(\OI\)-algebra over \(k\) for which \(\A(0)\)
  is not \(k\) but is still commutative, we can simply think of \(\A\)
  as an \(\OI\)-algebra over \(\A(0)\), which adheres to the
  definition given in \cite{1710.09247}. All of the examples we
  consider will meet this requirement.
\end{rk}

\begin{defn}
  If \(\A\) is an \(\OI\)-algebra over \(k\), an
  \(\A\)\textit{-module} is an \(\OI\)-module \(\M\) over \(k\) with
  the property that for any object \(w\) in \(\OI\), the \(k\)-module
  \(\M(w)\) has an \(\A(w)\)-module structure, and for which these
  module structures and the structure maps
  \(\M(w) \rightarrow \M(w')\) are coherent in the sense that, for any
  morphism \(\eps : w \rightarrow w'\), the following square commutes.

  \[
    \begin{tikzcd}
      \A(w) \otimes_{k} \M(w) \ar[r] \ar[d] & \M(w) \ar[d] \\
      \A(w')\otimes_{k} \M(w') \ar[r] & \M(w')
    \end{tikzcd}
  \]
  
  Here the horizontal maps come from the the \(\A(w)\)-module and
  \(\A(w')\)-module structures on \(\M(w)\) and \(\M(w')\),
  respectively, and the vertical maps are determined by the values of
  the functors \(\A\) and \(\M\) on \(\eps.\)
\end{defn}

\begin{rk}
  If we think of \(\A\) as a monoid in the category of
  \(\OI\)-algebras, then what we have just defined is equivalent to an
  \(\A\)-module in the category-theoretic sense \cite{maclane:71},
  though since we are focused on commutative or graded-commutative
  \(\OI\)-algebras, we don't need to specify whether the \(\A\) action
  is a left or right action.
\end{rk}

\begin{example}\label{constalg}
  The constant \(\OI\)-module \(\mathbf{k}\) discussed above in
  \cref{const} is in fact an \(\OI\)-algebra over \(k\), and an
  \(\OI\)-module over \(k\) is the same object as a
  \(\mathbf{k}\)-module.
\end{example}

The most natural examples of \(\OI\)-algebras come from composing an
OI-module with a functor from the category of \(k\)-modules to the
category of \(k\)-algebras.

\begin{example}\label{symext}
  Let \(\M\) be an \(\OI\)-module over \(k\) and let \(F\) be a
  functor from the category of \(k\)-modules to the category of
  \(k\)-algebras. Then composition \(F \circ \M\) is an
  \(\OI\)-algebra over \(k\). We will typically use the more familiar
  notation \(F(\M)\) for such constructions. The most important
  examples for our purposes will come from the symmetric algebra
  \(\operatorname{Sym}_{\bullet}({-})\) and the exterior algebra
  \(\bigwedge\nolimits^{\bullet}({-})\) functors.
\end{example}

The algebra \(\operatorname{Sym}_{\bullet}(\bF^{\OI,1})\) is exactly
the \(\OI\)-algebra \(\mathbf{X}^{\OI,1}\) described in
\cite{1710.09247}. More generally, applying the symmetric algebra
functor to a finitely generated free \(\OI\)-module \(\bF\) yields an
\(\OI\)-algebra \(\A\) for which \(\A(w)\) is a polynomial ring over
\(k\) for all \(w\), with predictable variable indices and structure
maps. \(\OI\)-algebras that arise from this construction will be our
focus for the rest of this paper, and we encode them via the following
definition.

\begin{defn}
  An \(\OI\)-algebra \(\A\) is called a \textit{polynomial}
  \(\OI\)\textit{-algebra over \(k\)} if it is isomorphic to
  \(\operatorname{Sym}_{\bullet}(\bF)\), where \(\bF\) is a finitely
  generated free \(\OI\)-module over \(k\). We will call a polynomial
  \(\OI\)-algebra \textit{standard-graded} if \(\bF\) is a graded
  \(\OI\)-module over \(k\) generated in degree \(1\). In this case
  each polynomial ring \(\A(w)\) is standard-graded.
\end{defn}

\begin{example}\label{standardex}
  Let \(\bF =\bigoplus\limits_{i = 1}^{n}\bF^{\OI,0}\) be a free
  \(\OI\)-module over \(k\) generated by \(n\) elements of width
  \(0\), i.e., \(\bF\) is the constant functor defined by
  \(\bF(w) = k^{n}\) for all widths \(w\). Then the \(\OI\)-algebra
  \(\A = \operatorname{Sym}_{\bullet}(\bF)\) is the constant
  OI-algebra defined by \(\A(w) = R = k[x_{1}, \ldots, x_{n}]\) for
  all \(w\), and \(\A\)-modules are simply \(\OI\)-modules over \(R.\)
\end{example}

\begin{example}
  Let \(\bF = \bF^{\OI,1}\) be a free \(\OI\)-module over \(k\) of
  rank \(\1\) and generated in width \(1\). Then the \(\OI\)-algebra
  \(\A = \operatorname{Sym}_{\bullet}(\bF)\) is the \(\OI\)-algebra
  \(\mathbf{X}^{\OI,1}\) from \cite{1710.09247}. In other words, in
  any width \(w\) we have
  \[
  \A (w) = k[x_{1}, \ldots, x_{w}],
  \]
  and an \(\OI\)-morphism \(\eps \colon w \rightarrow w'\) induces an
  algebra map \(\A(w) \rightarrow \A(w')\) defined by
  \(x_{i} \mapsto x_{\eps(i)}\). More generally, if
  \(\bF = \bigoplus\limits_{i = 1}^{n}\bF^{\OI,1}\) is a free
  \(\OI\)-module generated by \(n\) elements of width \(1\), then the
  \(\OI\)-algebra \(\mathbf{B} = \operatorname{Sym}_{\bullet}(\bF)\)
  is the \(\OI\)-algebra where for any width \(w\), \(\mathbf{B}(w)\)
  is the coordinate ring of the space of \(n \times w\) matrices,
  i.e.,
  \[
    \mathbf{B}(w) = k[~x_{ij}~|~1 \leqslant i \leqslant n ~\text{and}~
    1 \leqslant j \leqslant w~].
  \]
  A morphism \(\eps\) in \(\OI\) acts on an element of
  \(\mathbf{B}(w)\) by application on the second index,
  \(x_{i,j} \mapsto x_{i,\eps (j)}\). This is exactly the
  \(\OI\)-algebra \((\mathbf{X}^{\OI,1})^{\otimes c}\) described in
  \cite{1710.09247}.
  
\end{example}

\begin{example}
  If \(\bF = \bF^{\OI,1} \oplus \bF^{\OI,2}\), then
  \(\mathbf{C} = \operatorname{Sym}_{\bullet}(\bF)\) is a polynomial
  \(\OI\)-algebra with one variable of width \(1\) and one variable of
  width \(2\). Explicitly, the rings \(\mathbf{C}[w]\) for the first
  few values of \(w\) are the polynomial rings

  \begin{align*}
    \mathbf{C}[0] &= k\\
    \mathbf{C}[1] &= k[x_{1}]\\
    \mathbf{C}[2] &= k[x_{1}, x_{2}, y_{12}]\\
    \mathbf{C}[3] &= k[x_{1},x_{2},x_{3}, y_{12}, y_{13}, y_{23}]\\
    \mathbf{C}[4] &= k[x_{1},x_{2},x_{3},x_{4}, y_{12},y_{13},y_{14},y_{23},y_{24},y_{34}].
  \end{align*}
  
  The indices on the variables \(y_{ij} \in B_{n}\) represent the maps
  in \(\operatorname{Hom}_{\OI}(2,n)\) as described in
  \cref{oi}. Again, in the language of \cite{1710.09247} this is the
  algebra \(\mathbf{X}^{\OI,1} \otimes \mathbf{X}^{\OI,2}\)

\end{example}

\begin{defn}\label{freeamod}
  Let \(\A\) be a commutative \(\OI\)-algebra over \(k\). An
  \(\A\)-module is called \(\textit{free}\) if it is isomorphic to
  \(\A \otimes_{k} \bF\), where \(\bF\) is a free \(\OI\)-module over
  \(\mathbf{k}\). Following \cite{1710.09247}, we will use the
  notation \(\bF^{\OI,n}_{\A}\) for the free \(\A\)-module
  \(\A\otimes_{k}\bF^{\OI,n}\).
\end{defn}

\begin{rk}
  Any rank \(r\) free \(\A\)-module \(\bF\) has the form
  \({\displaystyle \bF = \bigoplus_{i = 1}^{r} \bF_{\A}^{\OI,n_{i}}
  }\).  There is a natural choice of basis
  \(\{e_{1}, \ldots, e_{r}\}\) just as in \cref{free}, and we note
  that any \(\A\)-module morphism \(\bF \rightarrow \M\) is determined
  by the image of this basis. To be more explicit, the free
  \(\A\)-module \(\bF_{\A}^{\OI,n}\) is, in width \(w\), a free
  \(\A(w)\) module with basis indexed by
  \(\operatorname{Hom}_{\OI}(n,w)\), and an \(\OI\)-morphism
  \(\eps : w \rightarrow w'\) induces a map
  \(\bF_{\A}^{\OI,n}(w) \rightarrow \bF_{\A}^{\OI,n}(w')\) which is
  defined on basis elements by post-composition with the indexing
  morphism, and which acts on coefficients via the structure map
  \(\A(\eps)\).
 \end{rk}

  \begin{lem}\label{freeproj}
    Any free \(\OI\)-module over a commutative \(\OI\)-algebra \(\A\)
    is projective.
  \end{lem}
  
  \begin{proof}
    The proof is nearly identical to the proof of the equivalent
    statement about ordinary \(R\)-modules. Let \(\bF\) be a free
    \(\A\)-module with basis \(\{e_{1}, \ldots, e_{r}\}\) whose
    elements have widths \(w_{1}, \ldots, w_{r}\), and consider the
    usual diagram
      \[
        \begin{tikzcd}
          & \bF \ar[d,"f"] \ar[dl, dotted, "\widetilde{f}"'] & \\
          \M \ar[r,"p"] & \mathbf{N} \ar[r] & 0,
        \end{tikzcd}
      \]
      where \(p\) is a surjective map of \(\A\)-modules. A lift
      \(\widetilde{f} : \bF \rightarrow \M\) is constructed by
      choosing, for each \(e_{i}\), any preimage
      \(m_{i} \in p^{-1}(f(e_{i})) \subseteq \M(w_{i})\) and defining
      \(\widetilde{f}(e_{i}) = m_{i}\).
  \end{proof}
  
\begin{defn}
  If \(\A\) is a commutative OI-algebra over \(k\), an \textit{ideal
    in \(\A\)} is a \(\A\)-submodule
  \(\Ib \colon \OI \rightarrow k\text{-mod}\) of \(\A\).  Thus, for
  each width \(w\), \(\Ib(w)\) is an ideal in \(\A(w)\), and for each
  \(\OI\)-morphism \(\eps \colon w \rightarrow w'\), the map
  \(\Ib(\eps)\) is the restriction of the ring map \(\A(\eps)\) to
  \(\Ib(w).\)
\end{defn}
  
\begin{example}
  Let \(\A\) be the OI-algebra
  \(\operatorname{Sym}_{\bullet}(\bF^{\OI,1})\), and let
  \(\Ib\) be the ideal in \(\A\) generated by \(x_{2}\)
  and \(x_{3}\) in \(\A(3)\), i.e., \(\Ib\) is the
  smallest \(\A\)-submodule of \(\A\) that contains
  these two width \(3\) elements. Then \(\Ib(w)\) is the zero
  ideal in \(\A(w)\) for \(w = 0,1,2\), While for any
  \(w \geqslant 3\) we have that \(\Ib(w)\) is the ideal
  generated by all of the variables of \(\A(w)\) except for
  \(x_{1}\), i.e.
  \(\Ib(w) = (x_{2}, \ldots , x_{w}) \subseteq \A(w).\)
\end{example}
  
We will primarily focus on polynomial \(\OI\)-algebras with variables
of width \(0\) and \(1\) only, i.e. algebras of the form
\(\operatorname{Sym}_{\bullet}(\bF)\) where \(\bF\) is a finitely
generated free \(\OI\)-module generated in widths \(0\) and \(1\). The
reason for this restriction is that such polynomial \(\OI\)-algebras
\(\A\) are precisely the \textit{noetherian} algebras (see
\cite[Proposition 4.8]{1710.09247}, in the sense that any ideal of
\(\A\) is finitely generated. If \(\A\) has just one variable in width
greater than \(1\), then this is no longer true.


\section{Minimality for OI Resolutions}\label{min}

We begin to develop a theory of minimal free resolutions for
\(\A\)-modules analogous to the theory of minimal free resolutions for
graded modules over a polynomial ring.

Recall that such minimal free resolutions can be characterized by the
minimality of the ranks of the free modules or, equivalently, by the
minimality of the maps appearing in a resolutions. One of our results
shows that both concepts of a minimality can be extended and lead to
equivalent descriptions of graded minimal free resolutions over a
polynomial \(\OI\)-algebra.  We also consider width-wise minimal and
flat resolutions.

Throughout this section we assume that \(\A\) is a noetherian
commutative OI-algebra over a field \(k\). If the category of
\(\A\)-modules is noetherian, then any finitely generated
\(\A\)-module \(\M\) admits a resolution by finitely generated free
\(\A\)-modules (see \cite[Theorem 7.1]{1710.09247}). In fact, since
\(\M\) is finitely generated, there exists a surjective map from a
finitely generated free \(\A\)-module \(\bF_{0}\) onto \(\M\). By
assumption, the kernel of this map is finitely generated . Thus, one
can iterate the construction and obtains a free resolution of \(\M\),
\[
  \cdots \xrightarrow{\varphi_{3}} \bF_{2} \xrightarrow{\varphi_{2}}
  \bF_{1} \xrightarrow{\varphi_{1}} \bF_{0} \xrightarrow{\varphi_{0}}
  \M \rightarrow 0
\]
where each module \(\bF_{d}\) is free and finitely
generated. Moreover, if \(\A\) and \(\M\) are graded, there is a graded
free resolution of \(\M\), that is, all maps \(\ffi_d\) are
degree-preserving.

In this section, we prove a family of results analogous to classical
results on minimal free resolutions for finitely generated graded
modules over polynomial rings.  Recall that in the classical setting,
a homogeneous map \(\varphi: F \rightarrow G\) between graded free
modules over a polynomial ring \(R\) is called \textit{minimal} if it
can be represented as a matrix having no unit entries. We adapt this
definition of minimality to maps between free \(\A\)-modules.

\begin{defn}\label{mindef}
  Let \(\A\) be a graded polynomial \(\OI\)-algebra, and let \(\bF\)
  and \(\mathbf{G}\) be finitely generated graded free \(\A\)-modules
  with bases \(\{f_{1}, \ldots, f_{s}\}\) and
  \(\{g_{1}, \ldots, g_{t}\}\), where each \(f_{i}\) has width
  \(u_{i}\) and each \(g_{j}\) has width \(v_{j}\). A (graded) map
  \(\varphi : \bF \rightarrow \mathbf{G}\) is determined by its value
  on the elements of the basis of \(\bF\), i.e., it is determined by
  \(s\) expressions of the form
  \[
    \varphi(f_{i}) = \sum\limits_{j = 1}^{t} \left(
      \sum\limits_{\eps_{i, j} \in \operatorname{Hom}(v_{j}, u_{i})}
      a_{\eps_{i, j} } \cdot \eps_{i, j} (g_{j})\right),
      \]
      where each \(a_{\eps_{i, j}}\) is a (homogenous) element of
      \(\A(u_{i})\). We say that \(\varphi\) is \textit{minimal} if
      whenever any map \(\eps_{i, j}\) is an identity morphism in
      \(\OI\), the coefficient \(a_{\eps_{i, j}}\) is not a unit. A
      graded complex of finitely generated free \(\A\)-modules is
      called \textit{minimal} if all of the maps in the complex are
      minimal.
\end{defn}

We will prove in \Cref{thm:char min res} that this definition of
minimality is equivalent to the definition one might expect, i.e.,
that a graded free resolution
\[
  \cdots \xrightarrow{\varphi_{3}} \bF_{2} \xrightarrow{\varphi_{2}}
  \bF_{1} \xrightarrow{\varphi_{1}} \bF_{0} \rightarrow
  \M \rightarrow 0
  \]
  of an \(\A\)-module \(\M\) is minimal if and only if for any other
  graded free resolution \(\mathbf{G}_{\bullet}\) of \(\M\) and for
  any homological degree \(d\), the rank of \(\bF_{d}\) is less than
  or equal to the rank of \(\mathbf{G}_{d}\). Furthermore, in analogy
  with the classical setting, any finitely generated graded
  \(\A\)-module admits a minimal free resolution, and any two minimal
  free resolutions of \(\M\) are isomorphic as chain complexes.

Observe that any graded free resolution \(\bF_{\bullet}\) of
\(\M\), 
\[
\cdots \xrightarrow{\varphi_{3}} \bF_{2}
  \xrightarrow{\varphi_{2}} \bF_{1} \xrightarrow{\varphi_{1}}
  \bF_{0} \xrightarrow{} \M \rightarrow 0 
  \]
  gives, for each width \(w\), a graded free resolution
  \(\bF_{\bullet} (w)\) of \(\M(w)\) over \(\A(w)\),
\[
  \cdots \rightarrow \bF_{2}(w) \rightarrow \bF_{1}(w) \rightarrow
  \bF_{0}(w) \rightarrow \M(w) \rightarrow 0.
  \]
  If \(\bF_{\bullet}\) is minimal one may wish that each resolution
  \(\bF_{\bullet} (w)\) is also minimal, but in general this is not
  the case. Requiring this property is a stronger, but slightly less well behaved
  condition on free resolutions, which we define here.

\begin{defn}\label{wwmindef}
  With the same setup as in \cref{mindef}, a map
  \(\varphi: \bF \rightarrow \mathbf{G}\) is called \textit{width-wise
    minimal} if none of the coefficients \(a_{\eps_{i, j}}\) is a
  unit, and a complex \(\bF_{\bullet}\) of free \(\A\)-modules is
  called width-wise minimal if all of the maps in the complex are
  width-wise minimal.
\end{defn}

It is clear from the definition that every width-wise minimal free
resolution is also a minimal free resolution. However, the converse is
not true in general. We illustrate this fact and the above concepts
with an example that is closely related to the Koszul complex defined
in \cite{1710.09247}.

\begin{example}\label{notwwmin}
  Let \(\A\) be the polynomial OI-algebra
  \(\operatorname{Sym}_{\bullet}(\bF^{\OI,1})\), and let
  \(\M \subseteq \bF^{\OI,1}_{\A}\) be the submodule which is
  generated by the element \(e_{2}\) in width \(2\), i.e.,
  \(\M(0) = \M(1) = 0,\) and, for \(w \geqslant 2\),
  \[
    \M(w) \subseteq \bF^{\OI,1}(w) ~\text{is generated by}~ \{e_2,
    e_3,\ldots, e_w\}.
  \]
  Since \(\M\) has a single generator in width \(2\), any minimal free
  resolution of \(\M\) has to start with a surjective map from
  \(\ffi_0 \colon\bF^{\OI,2} \rightarrow \M\) defined by sending the
  generator \(e_{12}\) of \(\bF^{\OI,2}\) to the generator \(e_{2}\)
  of \(\M\). In width \(2\), this map is injective. In width \(3\),
  this map is determined by
  \(e_{12} \mapsto e_2, \ e_{13} \mapsto e_3\) and
  \(e_{23} \mapsto e_3\). The kernel of the surjection
  \((\bF^{\OI,2}(3) \rightarrow \M (3)\) is generated by the element
  \(e_{13} - e_{23}\). In fact, one verifies, that the kernel of
  \(\ffi_0\) is the submodule of \(\bF^{\OI,2}\) that is generated by
  \(e_{13} - e_{23}\) in width \(3\).  Hence, using the map
  \(\ffi_1 \colon \bF^{\OI,3} \to \bF^{\OI,2}\) which sends the
  generator \(e_{123}\) to this kernel generator, we get the beginning
  of a free resolution
  \[
    \bF^{\OI,3} \xrightarrow{\varphi_{1}} \bF^{\OI,2}
    \xrightarrow{\varphi_{0}} \M \rightarrow 0.
  \]
  In the notation of \cref{mindef}, the map \(\ffi_1\) is determined
  by the expression
  \[
    \varphi_{1}(e_{123}) = e_{13} - e_{23},
  \]
  where the width \(3\) elements \(e_{13}\) and \(e_{23}\) are the
  images of the generator \(e_{12}\) of \(\bF^{\OI,2}\) under two
  different \(\OI\)-morphisms from \(2\) to \(3\). Both of their
  coefficients are units, but since none of the indexing morphisms is
  an identity morphism, \(\ffi_1\) is a minimal map. However, it is
  not width-wise minimal.
  
  Since \(\M\) is not free and \(\bF^{\OI,3}\) and \(\bF^{\OI,2}\)
  both have rank one, it makes sense to say that the constructed
  sequence is the beginning of a minimal free resolution of
  \(\M\). Notice though that this sequence does not give the beginning
  of a width-wise minimal free resolution because, for
  \(w \geqslant 3\), the module \(\M(w) \neq 0\) is a free
  \(\A(w)\)-module.
\end{example}

For our study of minimal resolutions, we first prove a relationship
between minimality and trivial complexes which is a direct analog of
the corresponding result for graded modules over a polynomial ring,
namely that the only way a complex can fail to be minimal is if it has
a trivial complex as a direct summand.

\begin{defn}
  A complex of free \(\A\)-modules is called \textit{trivial} if it is
  isomorphic to a direct sum of complexes of the form
  \[
    \cdots \rightarrow 0 \rightarrow 0 \rightarrow \bF
    \xrightarrow{\text{id}} \bF \rightarrow 0 \rightarrow 0
    \rightarrow \cdots,
  \] 
  where \(\bF\) is a finitely generated free \(\A\)-module.
\end{defn}

\begin{lem}\label{trivsummand}
  A graded complex \(\bF_{\bullet}\) of finitely generated free
  \(\A\)-modules is minimal if and only if it does not have a trivial
  complex as a direct summand.
\end{lem}

\begin{proof}
  It suffices to prove the result for a two-term complex
  \(\varphi \colon \bF \rightarrow \mathbf{G}\). With notation as in
  \cref{mindef}, suppose that \(\varphi\) is determined by
  \[
    \varphi(f_{i}) = \sum\limits_{j = 1}^{t} \left(
      \sum\limits_{\eps_{i, j} \in \operatorname{Hom}(v_{j}, u_{i})}
      a_{\eps_{i, j} } \cdot \eps_{i, j} (g_{j})\right),
  \]
  and assume without loss of generality that
  \(u_{1} = v_{1}\) and that \(a = a_{\eps_{1, 1} }\) is a unit in
  \(\A(u_{1})\). Since the identity map is the only element of $\operatorname{Hom}(v_{1}, u_{1})$ we get $\eps_{1, 1} = \text{id}$, and thus  
  \[
    \ffi(f_1) = a g_1 + \sum\limits_{j = 2}^{t} \left(
      \sum\limits_{\eps_{i, j} \in \operatorname{Hom}(v_{j}, u_{i})}
      a_{\eps_{i, j} } \cdot \eps_{i, j} (g_{j})\right). 
  \]
  Hence the set $\{g_1',\ldots,g_t'\}$ defined by   
  \[
    g_{1}' = \varphi(f_{1}) \quad \text{ and } \quad g_{j}' = g_{j} \;
    \text{ if } 2 \leqslant j \leqslant t
  \]
  is also a basis of \(\mathbf{G}\) . Similarly, we get another basis
  \(\{f_1',\ldots,f_s'\}\) of \(\bF\) by setting
  \[
    f_{1}' = f_{1} \quad \text{ and }  \quad  f_{i}' = f_{i} - a^{-1} \cdot \left( \sum\limits_{\eps_{i, 1} \in \operatorname{Hom}(v_{1}, u_{i})}
      a_{\eps_{i, 1} } \cdot \eps_{i, 1} (f_1)\right) \; \text{ if } 2 \leqslant i \leqslant s.  
  \]
  Since \(\ffi\) is a natural transformation, we have that
  \(\ffi (\eps_{i, 1} (f_1)) = \eps_{i, 1} (\ffi (f_1))\). If follows
  for \(i=2,\ldots,s\) that \(\varphi(f_{i}')\) has no summands
  involving \(\eps(g_{1})\) for any \(\OI\)-morphism \(\eps\). Since
  \(\varphi(f_{1}') = g_{1}'\), this shows that \(\ffi\) is a direct sum
  of two maps, at least one of which is trivial.
  
  Conversely, it is clear by definition that if \(\varphi\) has a
  trivial summand, then it is not minimal.
\end{proof}

Next, we prove a version of Nakayama's lemma for modules over
polynomial \(\OI\)-algebras, which states that for a finitely
generated \(\A\)-module, there is a well-behaved notion of a minimal
generating set.

\begin{lem}
  \label{oinak}
  Let \(\M \not = 0\) be a finitely generated graded \(\A\)-module. A
  subset \(S \subseteq \M\) of homogeneous elements is called a
  \emph{minimal generating set} of \(\M\) if \(S\) is a generating set
  for \(\M\) and no proper subset of \(S\) is a generating set for
  \(\M\). Let \(G = \{g_{1}, \ldots, g_{n}\}\) and
  \(H = \{h_{1}, \ldots, h_{m}\}\) be minimal generating sets for
  \(\M\). Then \(n = m\) and (up to a permutation), \(g_{i}\) and
  \(h_{i}\) have the same width and same degree for each \(i\) with
  \(1 \leqslant i \leqslant n\).
\end{lem}

\begin{proof}
  We proceed by induction on \(n\). If \(n = 1\), let \(w\) be the
  width of \(g_{1}\). Then \(\M(w)\) is minimally generated as
  an \(\A(w)\)-module by \(g_{1}\), and \(\M\) is
  generated by \(\M(w)\) as an \(\A\)-module. Since
  \(w\) is the smallest width for which \(\M(w)\) is non-zero,
  there must be at least one \(h_{i}\) of width \(w\). Let
  \(H' \subseteq H\) be the generators of width \(w\) and let
  \(H'' \subseteq H\) be the complement of \(H'\). Now \(H'\) is a
  minimal generating set for \(\M(w)\) as an
  \(\A(w)\)-module, and it follows that \(H'\) is a singleton
  set \(H = \{h_{1}\}\), and that the degree of \(h_{1}\) is the
  degree of \(g_{1}\). Since \(\M(w)\) generates
  \(\M\) as an \(\A\)-module and \(H\) is a minimal
  generating set, it follows that \(H''\) is empty.

  Next we consider the case where \(n > 1\). Let \(w\) be the smallest
  natural number for which \(\M(w)\) is nonzero. Then let
  \(G' \subseteq G\) and \(H' \subseteq H\) be the (necessarily
  nonempty) subsets consisting of the width \(w\) generators, and let
  \(G''\) and \(H''\) be their complements. Now \(G'\) and \(H'\) are
  both minimal generating sets for \(\M(w)\) as an
  \(\A(w)\)-module, so \(|G'| = |H'|\) and up to a permutation
  the elements of \(G'\) and \(H'\) have the same degrees. If we let
  \(\overline{\M}\) be the quotient of \(\M\) by the
  \(\A\)-submodule generated by \(\M(w)\), and let
  \(\overline{G''}\) and \(\overline{H''}\) be the images of \(G''\)
  and \(H''\) in \(\overline{M}\). Then \(\overline{G''}\) and
  \(\overline{H''}\) are minimal generating sets for
  \(\overline{\M}\). By induction, \(\overline{G''}\) and
  \(\overline{H''}\) have the same number of elements, and up to
  permutation these elements have the same widths and degrees. This
  lets us draw the same conclusion about \(G''\) and \(H''\), which
  completes the proof.
\end{proof}

\begin{rk}
  In the case that \(\bF\) is a finitely finitely generated generated
  free \(\A\)-module, the rank of \(\bF\) (see \cref{free}) is also
  the size of a minimal generating set for \(\bF\).
\end{rk}

Next we prove that a free resolution \(\bF_{\bullet}\) of \(\M\) is
minimal if and only if, when compared to any other graded free
resolution of \(\M\), it has the smallest number of generators in each
homological degree.

\begin{lem}\label{minequiv}
  Let \(\M\) be a finitely generated graded \(\A\)-module, and let
  \(\bF_{\bullet}\) be a graded resolution of \(\A\) by finitely
  generated free \(\A\)-modules. The following conditions are
  equivalent:

  \begin{enumerate}
  \item [(i)] \(\bF_{\bullet}\) is a graded minimal free resolution of
    \(\M\).
    
  \item [(ii)] If \(\mathbf{G}_{\bullet}\) is any other graded free
    resolution of \(\M\), then \(\mathbf{G}_{\bullet}\) contains
    \(\bF_{\bullet}\) as a direct summand.
 
  \item [(iii)] If \(\mathbf{G}_{\bullet}\) is any other graded free
    resolution of \(\M\), then for any homological degree \(d\), the
    rank of \(\bF_{d}\) is less than or equal to the rank of
    \(\mathbf{G}_{d}\)
  \end{enumerate}
\end{lem}

\begin{proof}

  To prove (i) \(\Rightarrow\) (ii), we construct a split surjective
  map of complexes \(\mathbf{G}_{\bullet} \rightarrow \bF_{\bullet}\)
  by induction on homological degree. In degree \(0\), the desired map
  \( \xi_0 \colon \mathbf{G}_{0} \rightarrow \bF_{0}\) exists because
  \(\mathbf{G}_{0}\) is free and hence projective by
  \cref{freeproj}. So the surjective map
  \(\varphi_{0} : \mathbf{G}_{0} \rightarrow \M\) lifts to a map
  \(\xi_{0}\), giving the following commutative diagram:
  \[
    \begin{tikzcd}
      \mathbf{G}_{0} \ar[r,"\varphi_{0}"] \ar[d,"\xi_{0}"] & \M\ar[d,"\text{id}_M"] \\
      \bF_{0} \ar[r,"\psi_{0}"'] & \M
    \end{tikzcd}
  \]
  
  To show that \(\xi_{0}\) is surjective, let \(f_{1}, \ldots, f_{n}\)
  be a homogeneous basis for \(\bF_{0}\), let
  \(m_{1}, \ldots, m_{n} \) be the images of these basis elements
  under \(\psi_{0}\), and consider the subset of
  $\{f_{1}, \ldots, f_{n}\}$ that is contained in the image of the map
  \(\xi_{0}\). This subset generates a submodule
  \(\bF' \subseteq \bF_{0}\).  If \(\bF'\) is a proper submodule of
  \(\bF_{0}\), this would imply that \(\M\) is generated by a proper
  subset of \(\{m_{1}, \ldots, m_{n}\}\).  Without loss of generality,
  this means we can express \(m_{1}\) in terms of
  \(\{m_{2}, \ldots, m_{n}\}\) as
  \[
    m_{1} = \sum\limits_{i=2}^{n} \left(\sum\limits_{\eps_i \in \operatorname{Hom}(u_{i}, u_{1})}a_{\eps_i} \eps_i (m_{i})\right), 
  \]
  where \(u_{i}\) is the width of \(m_{i}\) and each
  \(a_{\eps_i}\) is an element of \(\A(u_{i})\). Hence, the element 
  \[
    f_{1} - \sum\limits_{i=2}^{n} \left(\sum\limits_{\eps_i \in \operatorname{Hom}(u_{i}, u_{1})}a_{\eps_i} \eps_i (f_{i})\right)
  \]
  is in the kernel of \(\psi_{0}\) and so, by exactness, it is in the
  image of \(\psi_{1} \colon \bF_1 \to \bF_0\).  Since \(f_1\) is in
  the basis
  \(\{\pi (f_i) \; \mid \; 1 \le i \le n, \ \pi \in
  \operatorname{Hom}(u_{i}, u_{1}) \}\) of \(\bF_0 (u_1)\), it follows
  that \(\psi_{1}\) is not minimal.  This contradiction shows that
  \(\xi_{0}\) is surjective.  As \(\bF_{0}\) is free and therefore
  projective, \(\xi_{0}\) has a splitting
  \(s_{0} \colon \bF_{0} \rightarrow \mathbf{G}_{0}\) which embeds
  \(\bF_{0}\) as a direct summand of \(\mathbf{G}_{0}\).
  
  Now, suppose we have already built a split surjective map of chain
  complexes up to homological degree \(d-1\), i.e., we have surjective
  maps \(\xi_{i} : \mathbf{G}_{i} \rightarrow \bF_{i}\) for
  \(0 \leqslant i \leqslant d-1\) which form a map of truncated chain
  complexes, and each \(\xi_{i}\) has a section \(s_{i}\) which embeds
  \(\bF_{i}\) as a direct summand of \(\mathbf{G}_{i}\), and that
  these sections \(s_{i}\) also form a map of truncated chain
  complexes.
  
\[
\begin{tikzcd}
  \mathbf{G}_d \arrow[r, "\varphi_d"] & \mathbf{G}_{d-1} \arrow[r, "\varphi_{d-1}"] \arrow[d, "\xi_{d-1}"]       & \mathbf{G}_{d-2} \arrow[r] \arrow[d, "\xi_{d-2}"]          & \cdots \arrow[r] & \mathbf{G}_1 \arrow[r, "\varphi_1"] \arrow[d, "\xi_1"]       & \mathbf{G}_0 \arrow[r, "\varphi_0"] \arrow[d, "\xi_0"]       & \M \arrow[d, equal] \\
  \bF_d \arrow[r, "\psi_d"] & \bF_{d-1} \arrow[r, "\psi_{d-1}"]
  \arrow[u, "s_{d-1}", bend left] & \bF_{d-2} \arrow[r] \arrow[u,
  "s_{d-2}", bend left] & \cdots \arrow[r] & \bF_1 \arrow[r, "\psi_1"]
  \arrow[u, "s_1", bend left] & \bF_0 \arrow[r, "\psi_0"] \arrow[u,
  "s_0", bend left] & \M
\end{tikzcd}
\]

The image of the composition \(\xi_{d-1} \circ \varphi_{d}\) is
contained in the kernel of \(\psi_{d-1}\) and hence in the image of
\(\psi_{d}\). So by projectivity of \(\mathbf{G}_{d}\), we can lift
this composition to a map
\(\xi_{d} : \mathbf{G}_{d} \rightarrow \bF_{d}\). Using the minimality
of the map \(\psi_{d+1}\), an argument similar to the one given above
shows that $\xi_{d}$ is surjective. Repeating this argument, one gets
a surjective map of complexes
\(\xi_{\bullet} \colon \mathbf{G}_{\bullet} \rightarrow
\bF_{\bullet}\).

It remains to construct a splitting
\(s_{d} \colon \bF_{d} \rightarrow \mathbf{G}_{d}\) of \(\xi_{d+1}\)
with the property that
\(s_{d-1} \circ \psi_{d} = \varphi_{d} \circ s_{d}\). To this end,
consider the commutative diagram

\[
  \begin{tikzcd}
    \mathbf{G}_{d+1} \arrow[d, "\xi_{d+1}"] \arrow[r, "\varphi_{d+1}"] & \mathbf{G}_d \arrow[d, "\xi_{d}"] \arrow[r, "\varphi_{d}"]             & \mathbf{G}_{d-1} \arrow[d, "\xi_{d-1}"] \arrow[r, "\varphi_{d-1}"]       & \cdots \\
    \bF_{d+1} \arrow[r, "\psi_{d+1}"] & \bF_d \arrow[r, "\psi_{d}"]
    \arrow[u, "s_d", dotted, bend left] & \bF_{d-1} \arrow[u,
    "s_{d-1}", bend left] \arrow[r, "\psi_{d-1}"] & \cdots,
  \end{tikzcd}
\]
and let \(\{f_{1}, \ldots, f_{n}\}\) be a basis for \(\bF_d\).  For
each \(f_{i}\), the element \(s_{d-1} (\psi_{d}(f_{i})) \) is in the
kernel of
\(\varphi_{d-1}\). 
Let \(g_{i} \in \mathbf{G}_{d}\) be a preimage of
\(s_{d-1} (\psi_{d}(f_{i}))\) under \(\varphi_{d}\). Note that
\(\xi_{d}(g_{i}) - f_{i}\) is in the kernel of \(\psi_{d}\). Hence,
there is an element \(h_{i} \in \mathbf{G}_{d+1}\) with
\(\psi_{d+1} (\xi_{d+1} (h_i)) = \xi_{d}(g_{i}) - f_{i}\).  Define
\(s_{d}\) by sending \(f_{i}\) to \(g_{i} - \varphi_{d+1}(h_{i})\).
One verifies that \(s_{d}\) is a section of \(\xi_{d}\) and that
\(s_{d-1} \circ \psi_{d} = \varphi_{d} \circ s_{d}\).

The implication (ii) \(\Rightarrow\) (iii) follows because the rank of a free direct summand \(\bF\) of a free module \(\mathbf{G}\) is at most the rank of \(\mathbf{G}\). 

To show (iii) \(\Rightarrow\) (i), assume that \(\bF_{\bullet}\) is
not a minimal free resolution. Then by \cref{trivsummand} we can
decompose \(\bF_{\bullet}\) as
\(\mathbf{T}_{\bullet} \oplus \bF'_{\bullet}\), where \(\mathbf{T}\)
is a trivial complex and \(\bF'_{\bullet}\) is a free resolution of
\(\M\). In any homological degree \(d\) for which \(\mathbf{T}_{d}\)
is nonzero, the rank of \(\bF'_{d}\) is smaller than the rank of
\(\bF_{d}\), and so \(\bF_{\bullet}\) cannot be a direct summand of
$\bF'_{\bullet}$.  This contradiction completes the proof.
\end{proof}

\begin{lem}
\label{detbyrank}
A free \(\A\)-module \(\bF\) is determined up to isomorphism by the
set \(\{ \rank \bF(w) ~|~ w \in \N\}\), where \(\bF(w)\) is considered
as an \(\A(w)\)-module.
\end{lem}

\begin{proof}
  Since \(\bF\) and \(\mathbf{G}\) are free, we may assume
  \({\displaystyle \bF = \bigoplus_{i = 1}^{r} \bF_{\A}^{\OI,m_{i}}
  }\) and
  \({\displaystyle \mathbf{G} = \bigoplus_{j = 1}^{s}
    \bF_{\A}^{\OI,n_{j}}}\). The assumption gives for every integer
  \(w \geqslant 0\),
\[
  \rank \bF(w) = \sum_{i=1}^r \binom{w}{m_i} = \sum_{j=1}^s
  \binom{w}{n_j} = \rank \mathbf{G}(w).
\]   
Observe, for any integer \(t \geqslant 0\), that \(\binom{w}{t}\) is a
polynomial in \(w\) of degree \(t\). Thus, the claim follows by
comparing coefficients of the polynomials
\(\sum_{i=1}^r \binom{w}{m_i}\) and \(\sum_{j=1}^s \binom{w}{n_j}\).
\end{proof}

We collect the preceding results into the following theorem, which
draws an analogy between minimal free resolutions for modules over
noetherian polynomial rings and minimal free resolutions for modules
over \(\OI\)-algebras.

\begin{thm}\label{thm:char min res}
  Let \(\A\) be a polynomial \(\OI\)-algebra, and let \(\M\) be a
  finitely generated graded \(\A\)-module. Then one has:
  
  \begin{enumerate}
  \item [(i)] Any graded resolution of \(\M\) by finitely generated
    free \(\A\)-modules contains a minimal free resolution as a direct
    summand.
  \item [(ii)] Any two graded minimal free resolutions of \(\M\) are
    isomorphic.
  \item [(iii)] If \(\A\) is noetherian, then a graded minimal free
    resolution of \(\M\) exists.
  \end{enumerate}
  
\end{thm}

\begin{proof}
  Statement (i) is a consequence of \cref{minequiv}. For statement
  (ii), \cref{minequiv} implies that for two minimal free resolutions
  of \(\M\), each contains the other as a direct summand. Together
  with \cref{detbyrank} this means those resolutions are
  isomorphic. For statement (iii), noetherianity implies the existence
  of a resolution of \(\M\) by finitely generated free \(\A\)-modules
  (see \cite[Theorem 6.15]{1710.09247}), and applying statement (i)
  gives that any such resolution has a minimal free resolution as a
  direct summand.
\end{proof}

For explicit examples of graded minimal free resolutions we refer to
\Cref{cor:minimal free resolution} in the next section.

Although any graded module \(\M\) over a noetherian polynomial
\(\OI\)-algebra \(\A\) admits a graded minimal free resolution, such a
resolution is not necessarily width-wise minimal (see
\Cref{notwwmin}). However, as the following example shows, \(\M\) may
admit a graded complex of not necessarily free modules that in each
width \(w\) gives a minimal free resolution of \(\M(w)\) over
\(\A(w)\).

\begin{example}\label{wwkoszul}
  Let \(\A\) be the polynomial OI-algebra
  \(\operatorname{Sym}_{\bullet}(\bF^{\OI,1})\), and let \(\Ib\) be
  the ideal generated by \(x_{1} \in \A(2)\).

  As in \cref{notwwmin}, one can show that \(\Ib\) cannot have a
  width-wise minimal free resolution. However, as a special case of
  \cref{sss}, \(\Ib\) is resolved by graded exact chain complex
  \(\mathbf{B}_{\bullet}\) where in each width \(w\),
  \(\mathbf{B}_{\bullet}(w)\) is a minimal free resolution of
  \(\Ib (w)\).  We sketch its construction here. In each width \(w\),
  \(\Ib(w)\) is the ideal in \(\A(w) = k[x_{1}, \ldots, x_{w}]\)
  generated by the variables \(x_{1}, \ldots, x_{w-1}\), and so
  \(\Ib (w)\) is minimally resolved by a standard Koszul complex on
  these variables. Consider this Koszul complex as the graded
  components of the exterior algebra on a free \(\A(w)\)-module with
  basis \(\{e_{1}, e_{2}, \ldots, e_{w-1}\}\) and with the standard
  Koszul differentials. So \(\mathbf{B}_{i}(w)\) is the degree \(i\)
  component of this exterior algebra. These width-wise Koszul
  resolutions form a resolution \(\mathbf{B}_{\bullet}\) of \(\Ib\) by
  \(\A\)-modules, where an \(\OI\)-morphism \(\eps : w \rightarrow w'\)
  induces a map of exterior algebras
  \(\mathbf{B}_{i}(w) \rightarrow \mathbf{B}_{i}(w')\) by acting on
  indices \(e_{i} \mapsto e_{\eps(i)}\).

  The \(\mathbf{A}\)-module \(\mathbf{B}_{1}\) in this resolution is
  generated by the element \(e_{1}\) in width \(2\). But the rank of
  this module in any width \(w > 1\) is \(w - 1\). So it is not a free
  \(\A\)-module.
\end{example}

In order to extend the observation in the above example and to guide
the search for complexes that provide width-wise minimal free
resolutions we introduce flat \(\OI\)-modules.
  
\begin{defn}
  An \(\A\)-module \(\mathbf{Q}\) is called \textit{\(\A\)-flat} (or
  simply \textit{flat} when there is no ambiguity) if the functor
  \({-}\otimes_{\A}\mathbf{Q}\) is an exact functor from the category
  of \(\A\)-modules to itself.
\end{defn}

We now explain how this concept relates to width-wise minimal free
resolutions.

\begin{prop}
  \label{flatfree}
  Let \(\mathbf{Q}\) be a finitely generated graded module over a
  polynomial \(\OI\)-algebra \(\A\). Then \(\mathbf{Q}\) is
  \(\A\)-flat if and only if, for every width \(w\), \(\mathbf{Q}(w)\)
  is a free \(\A(w)\)-module.
\end{prop}

\begin{proof}
  Suppose \(\mathbf{Q}\) is \(\A\)-flat and consider
  \(\mathbf{Q}(w)\). Any sequence of \(\A(w)\)-modules
  \[
    0 \rightarrow N'' \rightarrow N \rightarrow N' \rightarrow 0
  \]
  can be lifted to an exact sequence of \(\A\)-modules, which is this
  sequence in width \(w\) and is the zero sequence in all other
  widths. Since \({-}_{\A}\otimes \mathbf{Q}\) is exact, it follows
  that \({-}\otimes_{\A(w)}\mathbf{Q}(w)\) is exact. Hence,
  \(\mathbf{Q}(w)\) is a finitely generated flat module over a
  polynomial ring \(\A(w)\), and so it is a projective
  \(\A(w)\)-module (see \cite{MR1731415}).  By the Quillen-Suslin
  Theorem \cite{Q, Suslin}, it follows that \(\mathbf{Q}(w)\) is a
  free \(\A(w)\)-module.
  
  Conversely, if \(\mathbf{Q}(w)\) is a free \(\A(w)\)-module for each
  width \(w\), then the functor \({-}\otimes_{\A}\mathbf{Q}\) is exact
  in each width, and so it is exact. Hence \(\mathbf{Q}\) is
  \(\A\)-flat.
\end{proof}

While free \(\A\)-modules are also \(\A\)-flat, the converse is of
course not true. Indeed, by the preceding result, the module \(\M \)
considered in \cref{notwwmin} is flat, but not free.

We now extend the notion of width-wise minimal resolutions to include
resolutions by flat modules.

\begin{defn}
  For an \(\A\)-module \(\M\), a \emph{width-wise minimal flat
    resolution of \(\M\)} is an exact sequence
\[ 
  \mathbf{Q}_{\bullet} \colon \cdots \rightarrow \mathbf{Q}_{3}
  \rightarrow \mathbf{Q}_{2} \rightarrow \mathbf{Q}_{1} \rightarrow
  \mathbf{Q}_{0} \rightarrow \M \rightarrow 0,
\]
where each \(\mathbf{Q}_{i}\) is a flat \(\A\)-module and, 
  in each width \(w\),  the complex \(\mathbf{Q}_{\bullet}(w)\) is a
  minimal free resolution of \(\M(w)\) as an
  \(\A(w)\)-module.
\end{defn}

\begin{example}
  In \cref{notwwmin}, we exhibited an \(\A\)-module \(\M\), which does
  not admit a width-wise minimal free resolution. However, we observed
  above that \(\M\) is actually a flat \(\A\)-module. Hence, it has a
  (trivial) width-wise minimal flat resolution of the form
  \[
    0 \rightarrow \mathbf{Q}_{0} \rightarrow \M \rightarrow 0,
  \]
  where \(\mathbf{Q}_{0} = \M\) and the map above is an
  isomorphism. We present non-trivial examples of width-wise minimal
  flat resolutions in the next section.
\end{example} 

As in the classical situation, width-wise minimality admits an alternate characterization. 

\begin{lem}
  Let \(\M\) be a finitely generated graded \(\A\)-module over a field
  \(k\), and let \(\mathbf{k}\) be the constant \(\OI\)-algebra
  determined by \(k\). Consider a resolution of \(\M\) by finitely
  generated flat graded \(\A\)-modules
  \[
    \mathbf{Q}_{\bullet} \colon \cdots \rightarrow \mathbf{Q}_{2}
    \xrightarrow{\partial_{2}} \mathbf{Q}_{1}
    \xrightarrow{\partial_{1}} \mathbf{Q}_{0}.
  \]
  The resolution \(\mathbf{Q}_{\bullet}\) is width-wise minimal if and
  only if in the complex
  \(\mathbf{Q}_{\bullet}\otimes_{\A}\mathbf{k}\), the maps
  \(\partial_{i}\otimes_{\A}\mathbf{k}\) are all zero.
\end{lem}

\begin{proof}
  This follows immediately from our definition of width-wise
  minimal. Indeed, the complex \(\mathbf{Q}_{\bullet}\) is width-wise
  minimal if and only if, for any width \(w\), the maps in the complex
  \[
  \cdots \rightarrow \mathbf{Q}_{2}(w)\otimes_{\A(w)}k
    \xrightarrow{\partial_{2}(w)\otimes_{\A(w)}k}
    \mathbf{Q}_{1}(w)\otimes_{\A(w)}k
    \xrightarrow{\partial_{1}\otimes_{\A(w)}k}
    \mathbf{Q}_{0}(w)\otimes_{\A(w)}k
    \] 
  are all zero (see, e.g., \cite[Lemma 19.4]{MR1322960}). These are precisely the width-wise maps
  in the complex
  \(\mathbf{Q}_{\bullet}\otimes_{\A}\mathbf{k}.\)
\end{proof}

The following result summarizes the above discussion and motives the search for flat resolutions. 

\begin{prop}
   \label{prop:char widthwise minimal}
Consider a graded exact sequence of finitely generated modules over a polynomial $\OI$-algebra \(\A\), 
\[ 
  \mathbf{Q}_{\bullet} \colon \cdots \rightarrow \mathbf{Q}_{3}
    \rightarrow \mathbf{Q}_{2} \rightarrow \mathbf{Q}_{1} \rightarrow
    \mathbf{Q}_{0} \rightarrow \M \rightarrow 0.  
\]
Then the following conditions are equivalent: 
\begin{itemize}

\item[(i)] For every width $w$, the restricted complex
\[ 
  \mathbf{Q}_{\bullet} (w) \colon \cdots \rightarrow \mathbf{Q}_{3} (w)
    \rightarrow \mathbf{Q}_{2}  (w) \rightarrow \mathbf{Q}_{1} (w) \rightarrow
    \mathbf{Q}_{0}(w)  \rightarrow \M (w) \rightarrow 0, 
\]
is a minimal free resolution of the  \(\A(w)\)-module $\M (w)$. 

\item[(ii)] $ \mathbf{Q}_{\bullet}$ is a width-wise minimal flat resolution of $\M$. 

\end{itemize}

\end{prop}

\begin{proof}
\Cref{flatfree} shows that (i) implies (ii). The converse is clear. 
\end{proof}

In the next section, we present constructions of finitely generated
width-wise minimal flat resolutions of certain classes of ideals.


\section{Constructing resolutions}\label{resolutions}

Using cellular resolutions, we explicitly construct classes of graded
minimal free resolutions and width-wise minimal flat resolutions.

First examples of such resolutions can be obtained from a Koszul
complex introduced in \cite[Lemma 8.3]{1710.09247}.  Given any
\(\OI\)-algebra \(\A\) and any width one element \(a \in \A(1)\),
there is a complex of free \(\A\)-modules
\[
  \cdots \rightarrow \bF^{\OI,d}_{\A}
  \xrightarrow{\varphi_{d}} \bF^{\OI,d-1}_{\A}
  \rightarrow \cdots \rightarrow \bF^{\OI,1}_{\A}
  \xrightarrow{\varphi_{1}} \bF^{\OI,0}_{\A} = \A \rightarrow
  0, 
\]
which in each width \(w\) is the classical Koszul complex on the set
of images of \(a\) in \(\A(w)\). In fact, if \(\A\) is graded and
\(a\) is homogeneous, the above complex is a graded complex. To state
this precisely, we need some notation.

For an integer \(k\) and a graded \(\OI\)-module \(\M\), we denote by
\(\M [k]\) the module with the same module structure as \(\M\), but
with an (internal) grading given by
\([(\M [k])(n)]_j = [\M(n)]_{j + k}\) for any integer \(j\).

\begin{example}
    \label{exa:Koszul}
Let $\Ib$ be the ideal of $\A =  \mathbf{X}^{\OI,1}$ generated by $x_1$. Then the Koszul complex associated to $x_1$ gives a graded and exact sequence (see \cite[Proposition 8.4]{1710.09247}), 
\[
\cdots \rightarrow \bF^{\OI,d}_{\A}[-d]
  \xrightarrow{\varphi_{d}} \bF^{\OI,d-1}_{\A}[-d+1]
  \rightarrow \cdots \rightarrow \bF^{\OI,1}_{\A}[-1]
  \xrightarrow{\varphi_{1}} \Ib\rightarrow
  0. 
\]
In each width $w$, it minimally resolves $\Ib (w) = (x_1,\ldots,x_w)$. Thus, it is a width-wise minimal graded free resolution of the ideal $\Ib$. 
\end{example}

In \cite{FN}, we generalize the construction of the mentioned Koszul
complex to the case where \(\A\) is an arbitrary commutative \(\OI\)
algebra, replacing the choice of \(a \in \A(1)\) with a choice of a
sequence \(a_{1}, \ldots, a_{r}\) of \(r\) elements in \(\A\) of
arbitrary widths, where this sequence is encoded via an \(\A\)-module
map \(\bF \rightarrow \A\) for \(\bF\) is a free \(\A\)-module of rank
\(r\) generated in the appropriate widths. In particular, if \(\A\) is
a polynomial OI-algebra over a field \(k\), then the constructed
complex gives a width-wise minimal graded free resolution of
\(\mathbf{k}\).

\begin{example}\label{exa:gen Koszul} 
  Consider the standard-graded polynomial OI-algebra
  \(\A =
  \operatorname{Sym}_{\bullet}\left(\bF^{\OI,2}\right)\). Thus,
  using the notation of \cref{oi}, one has
  \[
    \A(w) = k\left[x_{ij} ~|~ 1 \leqslant i < j \leqslant w \right].
  \]
  Let \(\mathfrak{m}\) be the ideal generated by \(x_{12} \in
  \A(2)\). Thus, \(\mathfrak{m}(w)\) is the graded maximal ideal of
  \(\A(w)\) for any width \(w \ge 2\). In this case, the functorially
  constructed Koszul complex in \cite{FN} gives a width-wise minimal
  graded free resolution of \(\mathfrak{m}\). Note that since \(\A\)
  is not noetherian, finitely generated free resolutions of \(\A\)-modules 
  are in general not guaranteed to exist.
  
\end{example}

We now construct explicit resolutions of classes of ideals in a
noetherian polynomial \(\OI\)-algebra. To this end we use the theory
of cellular resolutions for monomial ideals in noetherian polynomial
rings. We introduce suitable notation and recall needed concepts.

\begin{defn}\label{basiccombinatorics}
  Let \(P\) be a partially ordered set. For any \(a,b\) in \(P\), the \textit{interval} from \(a\) to \(b\), denoted \([a,b]\), is the set
  \[
  [a,b] = \{x \in P ~|~ a \leqslant x \leqslant b\}.
  \]

  An \textit{order ideal} in \(P\) is a nonempty subset
  \(I \subseteq P\) such that for any \(x, y\) in \(P\), if
  \(x \leqslant y\) and \(y \in I\), then \(x \in I\). (Some authors
  refer to this as a \textit{lower set} in \(P\).)
\end{defn} 

Often we will identify a monomial with its exponent vector and make
use of the following partially ordered sets.

\begin{defn}\label{specificcombinatorics}
  Fix positive integers \(d, n, m_{1}, \ldots, m_{d}\) with
  \(d \leqslant n\), and
  \(m_{1}\leqslant m_{2}\leqslant \ldots \leqslant m_{d}\), and
  consider the sets
\begin{align*}
\mathcal{P} & = \left\{(i_{1}, i_{2}, \ldots, i_{d})~|~ 1 \leqslant i_{1} <
      i_{2} < \ldots < i_{d} \leqslant n \right\} \subset [n]^d,  \\
\mathcal{Q} & = \left\{(i_{1}, i_{2}, \ldots, i_{d})~|~ 1 \leqslant i_{1}
      \leqslant i_{2} \leqslant \ldots \leqslant i_{d} \leqslant n
    \right\} \subset [n]^d,  \\
\mathcal{R} & = \left\{(i_{1}, i_{2}, \ldots, i_{d}) ~|~ i_{1} < i_{2} <
      \ldots < i_{d} ~\text{and}~1 \leqslant i_{j} \leqslant m_{j}
      ~\text{for each}~ j\right\}    \subset [m_1] \times \cdots \times [m_d]
\end{align*}
together with their componentwise (or Gale) partial ordering, defined by
  \[
  (i_{1}, \ldots, i_{d}) \leqslant_{\mathrm{Gale}} (j_{1}, \ldots,
    j_{d}) \quad \text{if}~i_{k} \leqslant j_{k} ~\text{for}~ k= 1,\ldots,d. 
    \]
 \end{defn}

 Note that there is a natural bijection between \(\mathcal{P}\) and
 the squarefree monomials of degree \(d\) in a polynomial ring
 \(R = k[x_{1}, x_{2},\ldots, x_{n}]\), given by
 $(i_{1}, \ldots, i_{d}) \mapsto x_{i_{1}}\cdots x_{i_{d}}$.
 Similarly, there is a bijection between \(\mathcal{Q}\) and the set
 of all monomials of degree \(d\) in \(R\).  Using these
 correspondences, we define the monomial ideals we are interested in,
 following \cite[Definitions 3.4 and 3.6]{MR2515766}.

\begin{defn}
  Let \(I\) be a monomial ideal generated by monomials of degree
  \(d\).

  (i) If \(I\) is in \(R\), it is called \textit{squarefree strongly
    stable} if its monomial generating set corresponds to an order
  ideal in \(\mathcal{P}\). It is said to be \textit{strongly} stable
  if its monomial generating set corresponds to an order ideal in
  \(\mathcal{Q}\).
  
  (ii) If \(I\) is generated by monomials in the set
  \(M = \{x_{1,i_{1}}x_{2,i_{2}}\cdots
  x_{d,i_{d}}~|~(i_{1},i_{2},\ldots, i_{d}) \in \mathcal{R}\}\), then
  we say that \(I\) is a \textit{Ferrers ideal} if its monomial
  generating set corresponds to an order ideal in \(\mathcal{R}.\)
\end{defn}

For simplicity, the following theorem in \cite{MR2515766} is stated for a
  squarefree strongly stable ideal. It is based on earlier work in \cite{MR2438922, MR2457403} and 
  applies analogously to strongly
  stable and to Ferrers ideals as well. For information on cellular resolutions we refer to \cite{MS}.

\begin{thm}[{\cite[Theorem 3.13]{MR2515766}}]
            \label{complexofboxes}
Consider a  squarefree strongly stable ideal \(I \subseteq R = k[x_{1}, \ldots, x_{n}]\) that is generated in degree \(d\). 
Let \(\Delta = \Delta^{n-1} \times \Delta^{n-1} \times \ldots \times
\Delta^{n-1}\) be the polyhedral cell complex obtained by taking the
product of \(d\) simplices on the set \([n]\).  Denote by
\(\mathcal{V}_{I}\) the set of vertices of \(\Delta\) that correspond
to the monomial generators of \(I,\)
\[
\mathcal{V}_{I} = \left\{ \{i_{1}\} \times \{i_{2}\} \times \cdots
      \times \{i_{d}\} ~\big|~ x_{i_{1}}x_{i_{2}}\cdots x_{i_{d}} \in
      I ~\text{and}~ i_{1} < \ldots < i_{d}\right\},
  \]
  
  Let \(C_{I}\) be the vertex-induced subcomplex of \(\Delta\)
  determined by the set \(\mathcal{V}_{I}\). Then \(C_{I}\) supports a
  linear minimal free resolution \(B_{\bullet}\) of \(R/I\), called
  the \emph{complex-of-boxes resolution},
  \[
B_{\bullet} \colon \quad 0 \to   B_{\delta}[-d+\delta -1]  \rightarrow \cdots \rightarrow B_{2} [-d-1] \rightarrow B_{1}[-d] \rightarrow B_{0} = R \to R/I \to 0, 
  \]
  where \(\delta\) is the dimension of \(C_{I}\), \(B_j\) is the free
  \(R\)-module with basis
  \(\{e_P \; \mid \; P \in C_I, \dim (P) = j-1\}\) and the
  differential is defined by
  \[
    \partial(e_{P}) = \sum\limits_{Q} \sgn
    (P,Q)\frac{m_{P}}{m_{Q}}e_{Q}.
    \]
    This sum runs over the codimension one faces \(Q\) of \(P\),
    \(\sgn (P,Q)\) is an incidence function determined by an
    orientation on \(\Delta\), and \(m_{P} \in R\) is the least common
    multiple of the monomials corresponding to the vertices of \(P\).
 \end{thm}

 Before applying this result to the study of \(\OI\)-ideals, we
 illustrate it by an example.

\begin{example}
  \label{cob}
  Consider the squarefree strongly stable ideal
  \(I = (x_{1}x_{2}, x_{1}x_{3}, x_{1}x_{4}, x_{2}x_{3}, x_{2}x_{4})
  \subseteq k[x_{1}, x_{2}, x_{3}, x_{4}]\).  \cref{complexofboxes}
  shows that \(R/I\) has a minimal free resolution supported on the
  polyhedral cell complex shown below:
  \[
    \begin{tikzpicture}[scale=6]
      \node (12) at (0,0) [label=below:\(1\times 2\)] {}; \node (13)
      at (1,0) [label=below:\(1\times 3\)] {}; \node (23) at (2,0)
      [label=below:\(2\times 3\)] {}; \node (14) at (0.5,0.866)
      [label=above:\(1\times 4\)] {}; \node (24) at (1.5,0.866)
      [label=above: \(2\times 4\)] {};
        
      \draw [black, fill=black!15] (0,0) -- (1,0) -- (0.5,0.866) --
      (0,0); \draw [black, fill=black!25] (1,0) -- (2,0) -- (1.5,
      0.866) -- (0.5, 0.866) -- (1,0);
        
      \draw (12.center) -- (13.center) node [midway,
      label=below:\(1\times 23\)] {}; \draw (13.center) -- (23.center)
      node [midway, label=below:\(12\times 3\)] {}; \draw (12.center)
      -- (14.center) node [midway, label=left:\(1\times 24\)] {};
      \draw (13.center) -- (14.center) node [midway,
      label=right:\(1\times 34\)] {}; \draw (23.center) -- (24.center)
      node [midway, label=right:\(2\times 34\)] {}; \draw (14.center)
      -- (24.center) node [midway, label=above:\(12\times 4\)] {};
        
      \node (t) at (0.5, 0.433) {\(1\times 234\)}; \node (p) at (1.25,
      0.433) {\(12\times 34\)};
    \end{tikzpicture}
  \]
  It has five vertices, six edges and two 2-dimensional faces.
  Explicitly, the complex-of-boxes resolution determined by this cell
  complex is
  \[
    0 \to R(-4)^{2} \xrightarrow{\varphi_{1}} R(-3)^{6}
    \xrightarrow{\varphi_{1}} R(-2)^{5} \xrightarrow{\varphi_{0}}
    R^{1} \to R/I \to 0,
  \]
  where the entries of the differentials with the stated bases have
  the following coordinate matrices:
  \begin{align*}
    \varphi_{0} & = \begin{bmatrix}
      x_{1}x_{2} & x_{1}x_{3} & x_{2}x_{3} & x_{1}x_{4} & x_{2}x_{4}
    \end{bmatrix}, \\
    \varphi_{1} & = \begin{bmatrix}
      x_{3}  & 0      & -x_{4} & 0      & 0      & 0      \\
      -x_{2} & -x_{2} & 0      & -x_{4} & 0      & 0      \\
      0      & x_{1}  & 0      & 0      & -x_{4} & 0      \\
      0      & 0      & x_{2}  & x_{3}  & 0      & -x_{2} \\
      0      & 0      & 0      & 0      & x_{3}  & x_{2}  \\
    \end{bmatrix}, \\
    \varphi_{2} & = \begin{bmatrix}
      x_{4}  & 0 \\
      0      & x_{4} \\
      x_{3}  & 0 \\
      -x_{2} & -x_{2} \\
      0      & x_{1} \\
      0      & -x_{3} \\
    \end{bmatrix}.          
  \end{align*}
\end{example}

The next step is to consider an \(\OI\)-ideal that is generated in one
width by a strongly stable monomial ideal. As preparation, we
establish the following combinatorial results.

\begin{lem}\label{technicallemma}
  Fix an integer \( d \geqslant 1\). For any integer
  \(w \geqslant d\), let \(\mathcal{P}_{w}\) be the set of
  strictly-increasing \(d\)-tuples of integers from \(1\) to \(w\),
  i.e.
\[ 
  \mathcal{P}_{w} = \left\{ \vec{a} = (a_{1}, a_{2},\ldots , a_{d})
    ~\big{|}~ 1 \leqslant a_{1} < a_{2} < \ldots < a_{d} \leqslant w
  \right\}.
\]
Each set \(\mathcal{P}_{w}\) is a poset under the componentwise
partial ordering as discussed in \cref{basiccombinatorics}. Any
\(\OI\)-morphism \(\eps : w \rightarrow w+1\) induces a map
\(\mathcal{P}_{w} \rightarrow \mathcal{P}_{w+1}\) by componentwise
application: \(\eps(\vec{a}) = (\eps(a_{1}), \ldots, \eps(a_{d}))\).

If \(I_{w}\) is an order ideal in \(\mathcal{P}_{w}\), then the set
\[
  I_{w+1} = \left\{\eps(\vec{a}) ~\big{|}~ \vec{a} \in
    I_{w}~\text{and}~\eps\in\operatorname{Hom}_{\OI}(w,w+1)\right\}
\]
is an order ideal in \(\mathcal{P}_{w+1}.\)
  
\end{lem}

\begin{proof}
  The minimum element of \(\mathcal{P}_{w}\) and of
  \(\mathcal{P}_{w+1}\) is \(\{1,2,\ldots,d\}\), which we denote by
  \(\widehat{0}\). Let \(\vec{m}_{1}, \ldots, \vec{m}_{r}\) be the
  maximal elements of \(I_w\). Denote by \(\1\) the all ones vector
  \((1,\ldots,1) \in \Z^d\). Then \(I_{w+1}\) is contained in the
  union of the intervals \([\widehat{0},\vec{m}_i + \1]\) with
  \(1 \leqslant i \leqslant r\). In order to show equality, consider
  any \(\vec{b} = (b_1,\ldots,b_d)\) in \(I_{w+1}\).  Thus, there is
  some maximal element \(\vec{m} = (m_1,\ldots,m_d)\) of \(I_w\) with
  \(\vec{b} \leqslant \vec{m} + \1\). We are done if can show that
  there is some \(\vec{a} = (a_1,\ldots,a_d) \leqslant \vec{m}\) with
  \(\vec{b} = \eps (\vec{a})\) for some
  \(\eps\in\operatorname{Hom}_{\OI}(w,w+1)\).

  If \(\vec{b} = \widehat{0}\), then \(a = \widehat{0}\) and the
  identity map have the desired property. If
  \(\vec{b} > \widehat{0}\), there is a unique integer
  \(k \geqslant 1\) such that \(b_i = i\) if \(i < k\) and
  \(b_k > k\). Set \(\vec{a} = (1,2,\ldots,k-1,b_k
  -1,\ldots,b_d-1)\). Note that \(\vec{a}\) is in \(\mathcal{P}_w\)
  and that \(\vec{a} \leqslant \vec{m}\) because
  \(\vec{b} \leqslant \vec{m} + \1\). Moreover, for
  \(\eps \in \operatorname{Hom}_{\OI}(w,w+1)\) with
\[
  \eps(i) = \begin{cases}
    i & \text{ if } 1 \leqslant i < b_k \\
    i+1 & \text{ if } b_k \leqslant i \leqslant w,
  \end{cases}
\]
we get \(\vec{b} = \eps (\vec{a})\), which completes the argument.
\end{proof}

\begin{cor}\label{cor:generated by strongly stable ideal}
  Let \(I \subseteq k[x_{1}, \ldots, x_{w_{0}}]\) be a squarefree
  strongly stable monomial ideal generated in degree \(d\). Consider
  the ideal \(\Ib\) of
  \(\A = \operatorname{Sym}_{\bullet}(\bF^{\OI,1})\) that is generated
  in width $w_0$ by \(I\). Then, for each width \(w \geqslant w_0\),
  the monomial ideal \(\Ib(w)\) is a squarefree strongly stable.
\end{cor}

\begin{proof}
  We use induction on \(w\). If \(w = w_0\), the claim is true by
  assumption on \(\Ib_{w_0} = I\). Consider any width \(w > w_0\). By
  \cite[Lemma 2.3]{1710.09247}, the ideal \(\Ib_w\) is equal to the
  width \(w\) component of the ideal of \(\A\) that is generated in
  width \(w-1\) by \(\Ib_{w-1}\). By induction, \(\Ib_{w-1}\) is
  squarefree strongly stable. Thus, \cref{technicallemma} shows that
  \(\Ib_w\) is a squarefree strongly stable monomial ideal as well.
\end{proof}

Our main result in this section shows that the complex-of-boxes
resolutions for each width-wise component of an ideal as above can be
given an \(\OI\)-structure. As a consequence, this produces a
width-wise minimal flat resolution of the ideal.

\begin{thm}
  \label{sss}
  Let \(\Ib\) be an \(\OI\)-ideal in
  \(\A= \operatorname{Sym}_{\bullet}(\bF^{\OI,1})\) generated in width
  \(w_0\) by a squarefree strongly stable monomial ideal
  \(I \subseteq \A(w_{0}) = k[x_{1}, \ldots, x_{w_{0}}]\) whose
  generators all have degree \(d\). Then \(\A/\Ib\) has a graded
  width-wise minimal flat resolution \(\mathbf{B}_{\bullet}\),
  \[
    \cdots \to  \mathbf{B}_i[-d-i+1] \to \B_{i-1}[-d-i+2]  \to \cdots \to \B_1[-d] \to \B_0 = \A \to \A/\Ib \to 0, 
  \]
  which in every width \(w \geqslant w_0\) restricts to the graded
  minimal free resolution of \(\A(w)/\Ib(w)\) given by the
  complex-of-boxes. 
\end{thm}

\begin{proof}
  By \cref{cor:generated by strongly stable ideal}, in each width
  \(w \geqslant w_0\), the ideal \(\Ib(w)\) is a squarefree strongly
  stable monomial ideal. Thus, the complex-of-boxes in
  \cref{complexofboxes} gives a graded minimal free resolution of
  \(\A(w)/\Ib (w)\). We use these resolutions to define the desired
  flat \(\A\)-modules \(\B_i\).

  Let \(\B_0 = \A\). In order to define \(\B_i\) for \(i > 0\), we
  first specify it as an \(\A(w)\)-module in each width \(w\) and then
  describe the structure maps between components of different widths.

  Consider any integer \(i > 0\). If \(0 \le w < w_0\), the ideal
  \(\Ib (w)\) is zero, and we set \(\B_i (w) = 0\).  If
  \(w \geqslant w_0\), we define \(\B_i (w)\) as the the free
  \(\A(w)\)-module appearing in homological degree \(i\) of the
  complex-of-boxes resolution of \(\A/\Ib (w)\) as described in
  \cref{complexofboxes}. In particular, it has a basis consisting of
  elements \(e_P\), where \(P\) is any \((i-1)\)-dimensional face of
  the polyhedral cell complex \(C_{\Ib (w)}\). Thus \(P\) is of the
  form
  \(P = \sigma_{1} \otimes \cdots \otimes \sigma_{d} \subset
  \Delta^{w-1} \times \cdots \times \Delta^{w-1}\), where
  \(\Delta^{w-1}\) is the simplex on the vertex set \([w]\). In the
  remainder of the proof we use
  \(\sigma_{1} \otimes \cdots \otimes \sigma_{d}\) to denote the basis
  element \(e_P\).
  
  Given an \(\OI\) morphism \(\eps : w \rightarrow w'\) and an element
  of the form
  \(a \cdot (\sigma_{1} \otimes \cdots \otimes \sigma_{d})\) with
  \(a \in \A(w)\), we define a homomorphism of \(k\)-modules
  \(\mathbf{B}_{i}(\eps) \colon \mathbf{B}_{i}(w) \rightarrow
  \mathbf{B}_{i}(w')\) by
  \[
  \mathbf{B}_{i}(\eps)
    \left(a \cdot (\sigma_{1} \otimes \cdots \otimes
      \sigma_{d})\right) = \A(\eps)(a) \cdot
    (\eps(\sigma_{1}) \otimes \cdots \otimes
    \eps(\sigma_{d})),
    \]
    i.e., we apply the structure morphism \(\A(\eps)\) to the
    coefficient \(a\), and we apply the \(\OI\) morphism \(\eps\) to
    the elements of the \(\sigma_{j}\). One verifies that these maps
    give \(\mathbf{B}_{i}\) the structure of an
    \(\A\)-module. Moreover, \(\B_i\) is \(\A\)-flat by
    \cref{flatfree}.

    Next, we define the desired differential \(\B_i \to \B_{i-1}\) in
    any width \(w\) as the differential in the complex-of-boxes
    resolution of \(\A(w)/\Ib (w)\). In order to show that these
    width-wise assignments give a morphism of \(\A\)-modules, it
    suffices to check (see \cite[Lemma 2.3]{1710.09247}) that, for any
    \(\eps \in \operatorname{Hom}_{\OI}(w,w+1)\), the following square
    commutes:
    \[
      \begin{tikzcd}
        \mathbf{B}_{i}(w) \arrow[r, "\partial"] \arrow[d, "\mathbf{B}_{i}(\eps)"'] & \mathbf{B}_{i-1}(w) \arrow[d, "\mathbf{B}_{i-1}(\eps)"]\\
        \mathbf{B}_{i}(w+1) \arrow[r, "\partial"] &
        \mathbf{B}_{i-1}(w+1).
      \end{tikzcd}
    \]
    To this end we claim there are orientations of the polyhedral cell
    complexes used in \cref{complexofboxes} such that the resulting
    differentials are determined by
 \begin{equation}
      \label{eq:sign on product}
      \partial (\sigma_{1} \otimes \ldots \otimes \sigma_{d}) = \sum
      \limits_{i = 1}^{d} (-1)^{\sum\limits_{j = 1}^{i-1}
        \mathrm{dim}(\sigma_{j})} \sigma_{1} \otimes \ldots \otimes
      \partial(\sigma_{i}) \otimes \ldots  \otimes \sigma_{d},  
 \end{equation}
 where \(\partial\) applied to a simplex
 \(\sigma = \{i_{0}, i_{1}, \ldots, i_{j}\}\) is defined by
\begin{equation}
    \label{eq:sign on simplex}
    \partial(\sigma) =  \sum \limits_{k = 0}^{j} (-1)^{j} x_{i_{k}} \cdot (\sigma \setminus i_{k}).     
  \end{equation}
  This can be verified directly. More conceptually, this follows
  because \eqref{eq:sign on simplex} corresponds to a choice of an
  orientation on the simplicial complex generated by a
  \((w-1)\)-dimensional simplex so that the resulting cellular
  resolution gives a Koszul complex 
  \[
    0 \to K_w \to \cdots \to K_1 \to K_0 \to 0.  
  \]
  Thus, \eqref{eq:sign on product} is induced from the differential of
  the cellular resolution supported on the polyhedral cell complex
  \(\Delta^{w-1} \times \ldots \times \Delta^{w-1}\) with \(d\)
  factors. The latter corresponds to the \(d\)-fold tensor product of
  the truncated Koszul complex
  \[
    0 \to K_w \to \cdots \to K_1 \to 0.  
  \]
  Note that the mentioned cellular resolution resolves the Ferrers
  ideal (see \cite[Theorem 3.13]{MR2515766})
  \[
    (x_{1,1},\ldots,x_{1,w}) \cdots (x_{d,1},\ldots,x_{d,w}) \subset
    K[x_{i, j} \; \mid \; i \in [d], j \in [w] ] \cong (\A
    (w))^{\otimes d}.
  \]
  
  We now verify the claimed commutativity by checking it for an
  arbitrary \(a \cdot \sigma_{1} \otimes \ldots \otimes \sigma_{d}\)
  with \(a \in \A(w)\) and a basis element
  \(\sigma_{1} \otimes \ldots \otimes \sigma_{d}\) of
  \(\mathbf{B}_{i}(w)\). If \(d = 1\) this is straightforward using
  Formula \eqref{eq:sign on simplex} and is, in fact, part of the
  statement in \Cref{exa:Koszul}. If \(d \geqslant 1\) we compute on the
  one hand
  \begin{align*}
    \hspace{2em}&\hspace{-2em} 
                  \mathbf{B}_{i-1}(\eps) \circ \partial ( a \cdot \sigma_{1}\otimes \sigma_{2} \otimes \cdots \otimes \sigma_{d}) 
    \\
                & =  \mathbf{B}_{i-1}(\eps) \left( \sum \limits_{j = 1}^{d} (-1)^{\sum\limits_{k = 1}^{j-1}\dim (\sigma_{k})} a \cdot  \sigma_{1} \otimes \cdots \otimes \partial(\sigma_{j})  \otimes \cdots \otimes \sigma_{d}\right) \\
                &= \sum \limits_{j=1}^{d} (-1)^{\sum \limits_{k = 1}^{j-1}\text{deg}(\sigma_{k})}  \A(\eps)(a) \cdot  \eps(\sigma_{1})\otimes \cdots \otimes \eps(\partial(\sigma_{j})) \otimes \cdots \otimes \eps (\sigma_{d}).  
  \end{align*}    
  On the other hand we get 
  \begin{align*}
    \hspace{2em}&\hspace{-2em} 
                  \partial \circ \mathbf{B}_{i}(\eps)(a \cdot \sigma_{1}\otimes \cdots \otimes \sigma_{d}) \\
                & = \partial \big (\A(\eps)(a) \cdot  \eps(\sigma_{1}) \otimes \cdots \otimes \eps(\sigma_{d}) \big )\\ 
                & = \sum \limits_{j=1}^{d} (-1)^{\sum \limits_{k = 1}^{j-1}\dim (\eps(\sigma_{k}))} \A(\eps)(a) \cdot \eps(\sigma_{1})\otimes \cdots \otimes \partial(\eps(\sigma_{j})) \otimes \cdots \otimes \eps (\sigma_{d}). 
  \end{align*}  
  Since \(\eps (\partial (\sigma_j)) = \eps (\partial (\sigma_j))\) by
  the argument for the case \(d=1\), we conclude
  \[
    \mathbf{B}_{i-1}(\eps) \circ \partial ( a \cdot \sigma_{1}\otimes
    \sigma_{2} \otimes \cdots \otimes \sigma_{d}) = \partial \circ
    \mathbf{B}_{i}(\eps)(a \cdot \sigma_{1}\otimes \cdots \otimes
    \sigma_{d}),
  \]
  as desired. Thus, \(\mathbf{B}_{\bullet}\) is a complex
  \(\A\)-modules. It is an exact sequence \(\mathbf{B}_{\bullet}\) of
  flat \(\A\)-modules because because it is width-wise exact by
  construction. This completes the proof.
\end{proof}

We single out a special case where the above flat resolution is in
fact a free resolution.

\begin{cor}\label{cor:minimal free resolution}
  If \(\Ib\) is the ideal of
  \(\A= \operatorname{Sym}_{\bullet}(\bF^{\OI,1})\) generated in width
  \(d\) by the monomial \(x_{1}x_{2}\cdots x_{d}\), then \(\A/\Ib\)
  has a width-wise minimal graded free resolution of the form
  \[
    \cdots \to \mathbf{B}_i[-d-i+1] \to \B_{i-1}[-d-i+2] \to \cdots
    \to \B_1[-d] \to \B_0 = \A \to \A/\Ib \to 0,
  \]
  where
  \(\B_i = \big( \bF^{\OI,d+i-1}_{\A} \big )^{\binom{d+i-2}{d-1}}\).
\end{cor} 

\begin{proof} 
  We use the flat resolution of \(\A/\Ib\) described in
  \Cref{sss}. For any width \(w\) and any homological
  degree\(i \geqslant 1\), a basis of \(\B_i (w)\) is given by an
  \((i-1)\)-dimensional face
  \(P = \sigma_{1} \otimes \cdots \otimes \sigma_{d}\) of
  \(C_{\Ib (w)}\).  Thus \(Q = \sigma_1 \cup \cdots \cup \sigma_d\) is
  a subset \(\{k_1 < k_2 < \cdots < k_{d+i-1}\}\) of \([w]\) with
  \(\sigma_j = \left\{k_l ~\big{|}~ |\sigma_1| + \cdots +
    |\sigma_{j-1}| < l \leqslant |\sigma_1| + \cdots + |\sigma_{j}|
  \right\}\). Hence \(P\) is determined by first choosing a
  \((d+i-1)\)-element subset \(Q\) of \([w]\). Each such choice can be
  identified with an \(\OI\)-morphism \(\pi \colon [d+i-1] \to [w]\)
  whose image is \(Q\). The set of all such maps \(\pi\) index a basis
  of \(\bF^{\OI,d+i-1}_{\A} (w)\).

  Second, given such a choice of \(Q\), a face
  \(P = \sigma_{1} \otimes \cdots \otimes \sigma_{d}\) with
  \(Q = \sigma_1 \cup \cdots \cup \sigma_d\) is determined by choosing
  the cardinalities of \(\sigma_1,\ldots,\sigma_{d-1}\) as
  \(|\sigma_1| + \cdots + |\sigma_{d}| = d+i-1\). Since every
  \(\sigma_j\) is non-empty, one must have
  \(1 \leqslant |\sigma_j| \leqslant i\). Thus, given \(Q\), there are
  \(\binom{d+i-2}{d-1}\) choices for the cardinalities of
  \(\sigma_1,\ldots,\sigma_{d-1}\). In width \((d + i - 1)\), these
  choices index a basis for \(\mathbf{B}_{i}\) as a free
  \(\A\)-module.
\end{proof}

We illustrate the above results by an explicit example.

\begin{example}
  Let \(\A = \operatorname{Sym}_{\bullet}(\bF^{\OI,1})\) be the
  polynomial OI-algebra with one variable of width \(1\). Consider the
  ideal \(\Ib\) of \(\A\) that is generated in width two by
  \(x_{1}x_{2}\) and the ideal \(\mathbf{J}\) of \(\A\) that generated
  in width three by \((x_{1}x_{2}, x_{1}x_{3})\).  The
  complex-of-boxes construction gives a resolution of \(\Ib\) by free
  \(\A\)-modules which is both minimal and width-wise minimal by
  \Cref{cor:minimal free resolution}.  It gives a width-wise minimal
  flat resolution of \(\mathbf{J}\) that is not a free resolution. The
  cell complexes supporting these resolutions for widths \(2\) through
  \(5\) are shown in the table below.
  
  The \(\OI\)-structure on the family of complex-of-boxes resolutions
  in each width corresponds to an \(OI\)-structure on their supporting
  polyhedral cell complexes supporting the resolution of \(\Ib (w)\)
  as a subcomplex of the cell complex supporting the resolution of
  \(\Ib (w+1)\). In the example, we can identify the two-dimensional
  cell complex supporting the free resolution of \(\Ib(4)\) as a
  subcomplex of the three-dimensional cell complex for \(\Ib(5)\) in
  five different ways, corresponding to the five \(\OI\) morphisms
  from \(4\) to \(5\). See \cref{tab:subcpx}.
  \begin{table}[h]
    \centering
    \begin{tabular}[h]{|c||c|c|}
      \hline
      width & \(\Ib\) & \(\mathbf{J}\) \\
      \hline
      \(w = 2\) &  \begin{tikzpicture}[scale=.8,baseline=(current bounding box.center)]
        
        \coordinate (up) at (0,1);
        \coordinate[label = below:$x_{1}x_{2}$] (12) at (0,0);
        \coordinate (down) at (0,-1);
        \node at (up) {};
        \node at (down) {};
        \node at (12)[circle,fill,inner sep=1pt]{};
      \end{tikzpicture} & \(\varnothing\) \\
      \hline 
      \(w = 3\) & \begin{tikzpicture}[scale=.8,baseline=(current bounding box.center)]
        \coordinate (up) at (0,1);
        \coordinate[label = below:$x_{1}x_{2}$] (12) at (-2.5,0);
        \coordinate[label = below:$x_{1}x_{3}$] (13) at (0,0);
        \coordinate[label = below:$x_{2}x_{3}$] (23) at (2.5,0);
        \coordinate (down) at (0,-1);
        
        \node at (up) {};
        \node at (down) {};
        \node at (12)[circle,fill,inner sep=1pt]{};
        \node at (13)[circle,fill,inner sep=1pt]{};
        \node at (23)[circle,fill,inner sep=1pt]{};
        
        \draw (12) -- (13) -- (23);
      \end{tikzpicture}  & \begin{tikzpicture}[scale=.8,baseline=(current bounding box.center)]
        
        \coordinate (up) at (0,1);
        \coordinate[label = below:$x_{1}x_{2}$] (12) at (-2.5,0);
        \coordinate[label = below:$x_{1}x_{3}$] (13) at (0,0);
        \coordinate[label = below:\color{gray}{$x_{2}x_{3}$}] (23) at (2.5,0);
        \coordinate (down) at (0,-1);
        
        \node at (up) {};
        \node at (down) {};
        \node at (12)[circle,fill,inner sep=1pt]{};
        \node at (13)[circle,fill,inner sep=1pt]{};
        \node at (23)[circle,fill=gray,inner sep=1pt]{};
        
        \draw (12) -- (13);
        \path[draw,densely dotted,color=gray] (13) -- (23);
      \end{tikzpicture}\\
      \hline
      \(w = 4\) & \begin{tikzpicture}[scale=.8,baseline=(current bounding box.center)]
        
        \coordinate (up) at (0,5.3301);
        \coordinate[label = below:$x_{1}x_{2}$] (12) at (-2.5,0);
        \coordinate[label = below:$x_{1}x_{3}$] (13) at (0,0);
        \coordinate[label = below:$x_{2}x_{3}$] (23) at (2.5,0);
        \coordinate[label = left:$x_{1}x_{4}$] (14) at (-1.25,2.165);
        \coordinate[label = right:$x_{2}x_{4}$] (24) at (1.25,2.165);
        \coordinate[label = above:$x_{3}x_{4}$] (34) at (0,4.3301);
        \coordinate (down) at (0,-1);
        
        \node at (up) {};
        \node at (down) {};
        \node at (12)[circle,fill,inner sep=1pt]{};
        \node at (13)[circle,fill,inner sep=1pt]{};
        \node at (23)[circle,fill,inner sep=1pt]{};
        \node at (14)[circle,fill,inner sep=1pt]{};
        \node at (24)[circle,fill,inner sep=1pt]{};
        \node at (34)[circle,fill,inner sep=1pt]{};
        
        \draw (12) -- (13) -- (14) -- (12);
        \filldraw[opacity=0.15] (12) -- (13) -- (14) -- (12);
        
        \draw (13) -- (23) -- (24)-- (14)-- (13);
        \filldraw[opacity=0.15] (13) -- (23) -- (24)-- (14)-- (13);
        
        \draw (14) -- (24) -- (34) -- (14);
        \filldraw[opacity=0.15] (14) -- (24) -- (34) -- (14);
        
      \end{tikzpicture}  & \begin{tikzpicture}[scale=.8,baseline=(current bounding box.center)]
        
        \coordinate (up) at (0,5.3301);
        \coordinate[label = below:$x_{1}x_{2}$] (12) at (-2.5,0);
        \coordinate[label = below:$x_{1}x_{3}$] (13) at (0,0);
        \coordinate[label = below:$x_{2}x_{3}$] (23) at (2.5,0);
        \coordinate[label = left:$x_{1}x_{4}$] (14) at (-1.25,2.165);
        \coordinate[label = right:$x_{2}x_{4}$] (24) at (1.25,2.165);
        \coordinate[label = above:\color{gray}{$x_{3}x_{4}$}] (34) at (0,4.3301);
        \coordinate (down) at (0,-1);
        
        \node at (up) {};
        \node at (down) {};
        \node at (12)[circle,fill,inner sep=1pt]{};
        \node at (13)[circle,fill,inner sep=1pt]{};
        \node at (23)[circle,fill,inner sep=1pt]{};
        \node at (14)[circle,fill,inner sep=1pt]{};
        \node at (24)[circle,fill,inner sep=1pt]{};
        \node at (34)[circle,fill=gray,inner sep=1pt]{};
        
        \draw (12) -- (13) -- (14) -- (12);
        \filldraw[opacity=0.15] (12) -- (13) -- (14) -- (12);
        
        \draw (13) -- (23) -- (24)-- (14)-- (13);
        \filldraw[opacity=0.15] (13) -- (23) -- (24)-- (14)-- (13);
        
        \path[draw,densely dotted,color=gray] (14) -- (34);
        \path[draw,densely dotted,color=gray] (24) -- (34);
      \end{tikzpicture} \\
      \hline
      \raisebox{3cm}{\(w = 5\)} & \includegraphics[width=7cm]{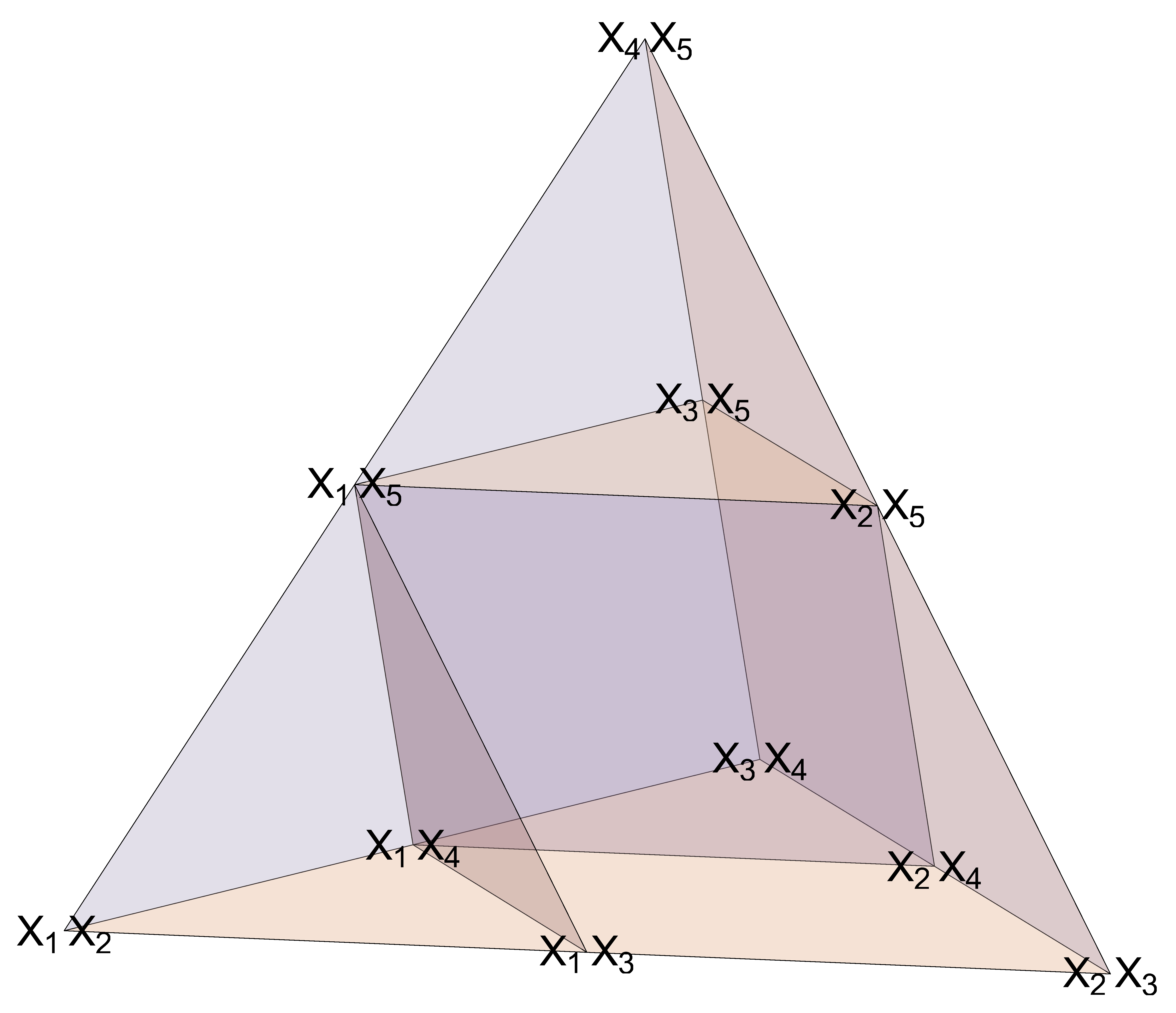} & \includegraphics[width=7cm]{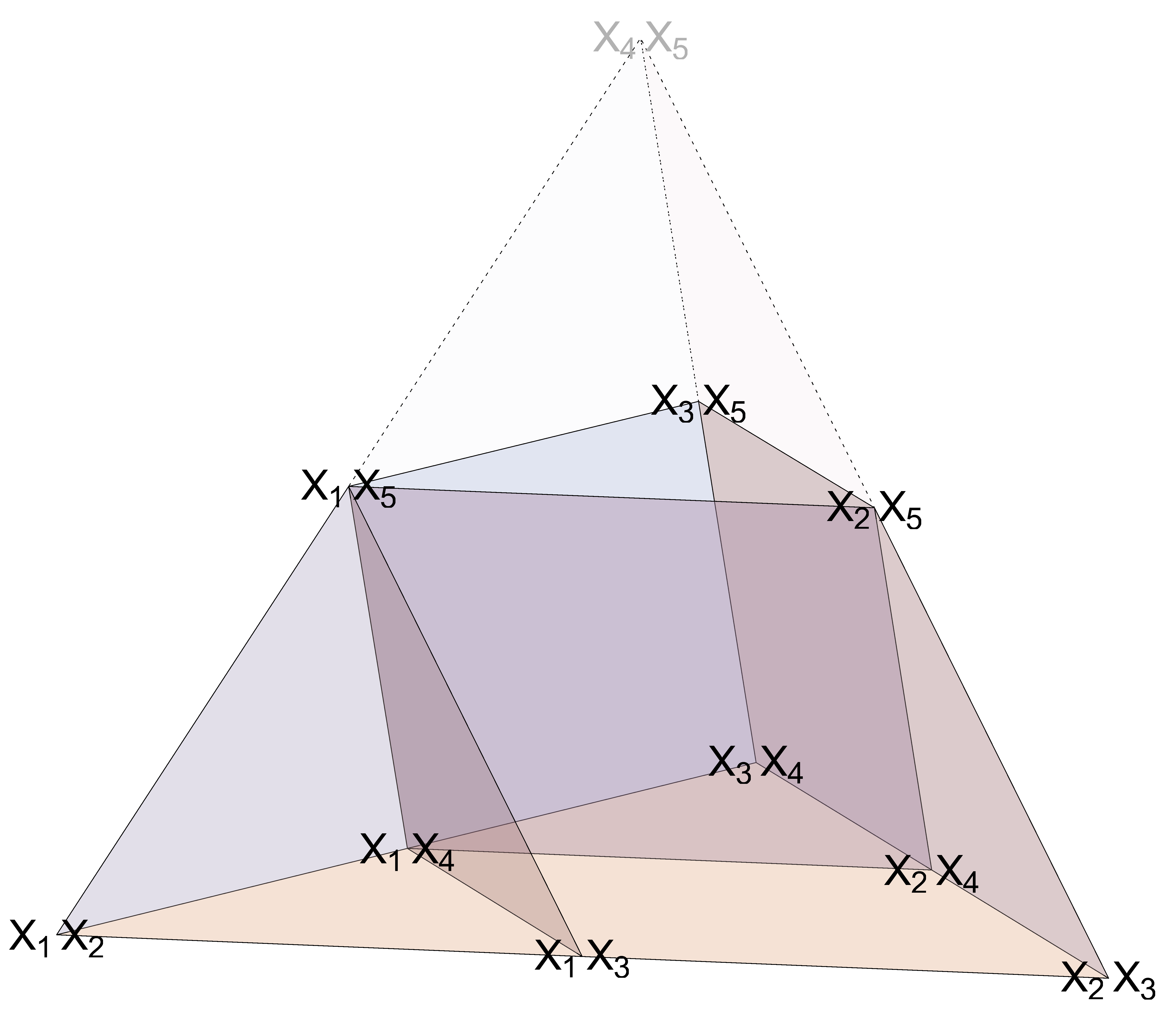}\\
      \hline
    \end{tabular}
    \caption{Simplicial complexes supporting width-wise minimal resolutions.}
    \label{tab:subcpx}
  \end{table}
\end{example}

As mentioned above, each squarefree strongly stable monomial ideal has
a corresponding Ferrers ideal, which also admits a linear minimal free
resolution supported on an identical cell complex. In fact, \cref{sss}
can be adapted to provide an analogous result for families of Ferrers
ideals.

\begin{thm}\label{ferrers}
  Let \(\A = \operatorname{Sym}_{\bullet}(\bF^{\OI,1})^{\otimes d}\)
  be the polynomial OI-algebra with \(d\) variables of width \(1\),
  that is, for every \(w\),
  \[
    \A(w) = k[x_{i,j} ~|~ 1 \leqslant i \leqslant d ~\text{and}~ 1
    \leqslant j \leqslant w]
  \]
  and \(\OI\) morphisms act on the second index of the variables. Let
  \(I \subseteq \A (w_0)\) be any Ferrers ideal generated in degree
  \(d\), and let \(\Ib\) be the ideal in \(\A\) which is generated in
  width \(w_0\) by \(I.\)

  Then in each width \(w \geqslant w_0\), the monomial ideal
  \(\Ib(w)\) is a Ferrers ideal and \(\A/\Ib\) has a graded width-wise
  minimal flat resolution \(\mathbf{B}_{\bullet}\),
  \[
    \cdots \to  \mathbf{B}_i[-d-i+1] \to \B_{i-1}[-d-i+2]  \to \cdots \to \B_1[-d] \to \B_0 = \A \to \A/\Ib \to 0, 
  \]
  which in every width \(w \geqslant w_0\) restricts to the graded
  minimal free resolution of \(\A(w)/\Ib(w)\) given by the
  complex-of-boxes.
\end{thm}

\begin{proof}
  The proof proceeds exactly as in the proof of \cref{sss} with only
  cosmetic changes. We leave the details to the interested reader. 
\end{proof} 

\begin{rk}
  It is worth noting that in \cite{MR2515766}, all three classes of
  monomial ideals mentioned in \cref{specificcombinatorics} admit
  complex-of-boxes resolutions: squarefree strongly stable monomial
  ideals, strongly stable monomial ideals, and Ferrers
  ideals. \cref{sss} and \cref{ferrers} show that, for
  \(\OI\)-paremetrized families of strongly stable ideals and Ferrers
  ideals, the complex-of-boxes resolutions can be given an
  \(\OI\)-structure. The analogous result is not true for strongly
  stable monomial ideals because in this case the analog of
  \Cref{cor:generated by strongly stable ideal} is not true, i.e., the
  widthwise components of an \(\OI\)-ideal generated by a strongly
  stable ideal in a fixed width are not necessarily strongly stable
  monomial ideals themselves.
  
  For example, if \(\Ib\) is the ideal of
  \(\A = \operatorname{Sym}_{\bullet}(\bF^{\OI,1})\) that is generated
  in width one by the strongly stable monomial ideal \((x_{1}^{3})\),
  then in width \(2\) we have
  \[
    \Ib(2) = (x_{1}^{3}, x_{2}^{3}).
  \]
  This ideal is not strongly stable since it is missing
  \(x_{1}^{2}x_{2}\) and \(x_{1}x_{2}^{2}\).
\end{rk}


 
\end{document}